\documentclass[11pt, a4paper]{amsart}
\usepackage{amssymb}
\usepackage{graphicx}
\newtheorem{lemm}{Lemma}[section]
\newtheorem{coro}{Corollary}[section]
\newtheorem{prop}{Proposition}[section]
\newtheorem{theo}{Theorem}[section]
\theoremstyle{definition}
\newtheorem{defi}{Definition}[section]
\newtheorem{rema}{Remark}[section]

\numberwithin{equation}{section}

\def\Bbb{\mathbb}
\def\fb{\partial\{u>0\}}
\def\uep{u^\varepsilon}
\def\ep{\varepsilon}
\def\pep{P_\varepsilon}

\def\a{\alpha}

\def\fint{\operatorname {--\!\!\!\!\!\int\!\!\!\!\!--}}

\def\uepj{u^{\varepsilon_j}}

\def\rn1{\Bbb R^{N+1}}
\def\rn{\Bbb R^{N}}

\def\epj{\varepsilon_j}

\def\fepj{f^{{\varepsilon}_j}}
\def\fep{f^{{\varepsilon}}}

\def\a*{\alpha^*_\ep}

\def\di{\displaystyle}
\def\R{\mathbb R}

\def\lstar{\lambda^*}
\def\lone{\lambda_{\min}}
\def\ltwo{\lambda_{\max}}

\def\pint{\operatorname {--\!\!\!\!\!\int\!\!\!\!\!--}}
\def\K{{\mathcal{K}}}
\def\psubep{p_{\varepsilon}}
\def\psubepj{p_{\epj}}
\begin{document}
\title[Inhomogeneous minimization problems for the $p(x)$-Laplacian]{Inhomogeneous minimization problems \\ for the $p(x)$-Laplacian}
\author[Claudia Lederman]{Claudia Lederman}
\author[Noemi Wolanski]{Noemi  Wolanski}
\address{IMAS - CONICET and Departamento  de
Ma\-te\-m\'a\-ti\-ca, Facultad de Ciencias Exactas y Naturales,
Universidad de Buenos Aires, (1428) Buenos Aires, Argentina.}
\email[Claudia Lederman]{clederma@dm.uba.ar} \email[Noemi
Wolanski]{wolanski@dm.uba.ar}
\thanks{Supported by the Argentine Council of Research CONICET under the project PIP 11220150100032CO 2016-2019,   UBACYT 20020150100154BA and
ANPCyT PICT 2016-1022.}

%35R35 Free boundary problems for PDE
%35B65 Smoothness and regularity of solutions of PDE
%35J60 Nonlinear elliptic equations
%35J70 Degenerate elliptic equations
%35J20 Variational methods for second-order, elliptic equations
%49K20 Calculus of variations: problems involving partial differential equations
%%%%%%%%%% abstract %%%%%%%%%%%%%%%%%%%%%%%%%%%%%%%%%%%%%
\begin{abstract} This paper is devoted to the study of inhomogeneous minimization problems associated to the $p(x)$-Laplacian. We make a thorough analysis of the
essential  properties of their minimizers and we establish a relationship with a suitable free boundary problem.

On the one hand, we study the problem of minimizing the functional
$J(v)=\int_\Omega\Big(\frac{|\nabla
v|^{p(x)}}{p(x)}+\lambda(x)\chi_{\{v>0\}}+fv\Big)\,dx$. We show
that nonnegative local minimizers $u$ are solutions to the
free boundary problem: $u\ge 0$ and
\begin{equation}
\label{fbp-px}\tag{$P(f,p,{\lambda}^*)$}
\begin{cases}
\Delta_{p(x)}u:=\mbox{div}(|\nabla u(x)|^{p(x)-2}\nabla
u)= f & \mbox{in }\{u>0\}\\
u=0,\ |\nabla u| = \lambda^*(x) & \mbox{on }\partial\{u>0\}
\end{cases}
\end{equation}
with
$\lambda^*(x)=\Big(\frac{p(x)}{p(x)-1}\,\lambda(x)\Big)^{1/p(x)}$
 and that the free boundary is a
$C^{1,\alpha}$ surface with the exception of a subset of
${\mathcal H}^{N-1}$-measure zero.

On the other hand, we study  the  problem of minimizing  the
functional $J_{\ep}(v)=\di \int_\Omega \Big(\frac{|\nabla
v|^{\psubep(x)}}{\psubep(x)}+B_{\ep}(v)+\fep v\Big)\, dx$, where
$B_\ep(s)=\int _0^s\beta_\ep(\tau) \, d\tau$, $\ep>0$,
${\beta}_{\varepsilon}(s)={1 \over \varepsilon} \beta({s \over
\varepsilon})$, with $\beta$  a  Lipschitz  function satisfying
$\beta>0$ in $(0,1)$, $\beta\equiv 0$ outside $(0,1)$.  We prove
that if $\uep$ are nonnegative local minimizers, then $\uep$ are
solutions to
\begin{equation}
\label{eq}\tag{$P_\ep(\fep, p_\ep)$}
\Delta_{p_\ep(x)}\uep={\beta}_{\varepsilon}(\uep)+\fep, \quad
u^{\ep}\geq 0.
\end{equation}

Moreover, if the functions $\uep$, $\fep$ and $p_\ep$ are
uniformly bounded, we show that limit functions $u$ ($\ep\to 0$) are solutions
to the free boundary problem $P(f,p,{\lambda}^*)$ with
$\lambda^*(x)=\Big(\frac{p(x)}{p(x)-1}\,M\Big)^{1/p(x)}$, $M=\int
\beta(s)\, ds$, $p=\lim p_\ep$, $f=\lim \fep$, and that the free boundary is a $C^{1,\alpha}$
surface with the exception of a subset of ${\mathcal
H}^{N-1}$-measure zero.

In order to obtain our results we need to overcome deep technical difficulties and develop new strategies, not present in the previous literature for this type of problems.\end{abstract}
\maketitle

\bigskip

%%%%%%%%%%%%%% end of abstract %%%%%%%%%%%%%%%%%%%%%%%%%%%%%%%%%%%%%%%%%

%%%%%%%%%  introduction %%%%%%%%%%%%%%%%%%%%%%%%%%%%%%%%%%%%%%%%%%%%%%
\begin{section}{Introduction}
\label{sect-intro}

This paper is devoted to the study of inhomogeneous minimization problems associated to the $p(x)$-Laplacian. We make a thorough analysis of the
essential  properties of their minimizers and we establish a relationship with a suitable free boundary problem.

The first  minimization problem under consideration corresponds to the functional
\begin{equation}\label{funct-J}
J(v)=\int_\Omega\Big(\frac{|\nabla
v|^{p(x)}}{p(x)}+\lambda(x)\chi_{\{v>0\}}+fv\Big)\,dx.
\end{equation}
In the particular case in which $p(x)\equiv 2$ and $f(x)\equiv 0$, the functional becomes
\begin{equation*}
\int_\Omega\Big(\frac{|\nabla
v|^2}{2}+\lambda(x)\chi_{\{v>0\}}\Big)\,dx.
\end{equation*}
The corresponding minimization problem  in $H^1(\Omega)$ with prescribed nonnegative values on $\partial\Omega$ was first treated by Alt and Caffarelli  in the seminal paper \cite{AC} motivated by the study of flow problems of jets and cavities. In \cite{AC} it was shown that local minimizers are solutions of the following free boundary problem: $u\ge 0$ and
\begin{equation*}
\begin{cases}
\Delta u= 0 & \mbox{in }\{u>0\}\\
u=0,\ |\nabla u| = \lambda^*(x) & \mbox{on }\partial\{u>0\},
\end{cases}
\end{equation*}
with $\lambda^*(x)=(2\lambda(x))^{1/2}$ and   that the free boundary $\partial\{u>0\}$ is a
$C^{1,\alpha}$ surface with the exception of a subset of
${\mathcal H}^{N-1}$-measure zero.

\medskip

In the present work we prove that nonnegative local minimizers of functional \eqref{funct-J} are solutions to
the
inhomogeneous free boundary problem for the $p(x)$-Laplacian: $u\ge 0$ and
\begin{equation}
\tag{$P(f,p,{\lambda}^*)$}
\begin{cases}
\Delta_{p(x)}u:=\mbox{div}(|\nabla u(x)|^{p(x)-2}\nabla
u)= f & \mbox{in }\{u>0\}\\
u=0,\ |\nabla u| = \lambda^*(x) & \mbox{on }\partial\{u>0\},
\end{cases}
\end{equation}
with
$\lambda^*(x)=\Big(\frac{p(x)}{p(x)-1}\,\lambda(x)\Big)^{1/p(x)}$.

The $p(x)$-Laplacian serves as a model for a stationary
non-newtonian fluid with properties depending on the point in the
region where it moves. For example, such a situation corresponds
to an electrorheological fluid. These are fluids such that their
properties depend on the magnitude of the electric field applied
to it. In some cases, fluid and Maxwell's equations become
uncoupled and a single equation for the $p(x)$-Laplacian appears
(see \cite{R}).

\medskip

The second minimization problem we deal with corresponds to the functional
\begin{equation}\label{funct-Jep}
 J_{\ep}(v)=\di \int_\Omega \Big(\frac{|\nabla
v|^{\psubep(x)}}{\psubep(x)}+B_{\ep}(v)+\fep v\Big)\, dx,
\end{equation}
where $B_\ep(s)=\int _0^s\beta_\ep(\tau) \, d\tau$,
$\ep>0$, ${\beta}_{\varepsilon}(s)={1 \over \varepsilon} \beta({s
\over \varepsilon})$, with $\beta$  a  Lipschitz  function
satisfying $\beta>0$ in $(0,1)$, $\beta\equiv 0$ outside $(0,1)$.

The minimization problem for functional \eqref{funct-Jep} is a regularization of the one corresponding to functional \eqref{funct-J}. The primary purpose in studying a regularized problem is to obtain uniform properties and establish results which carry over in the limit. In fact, we  prove that if  $\uep$ are nonnegative local minimizers to \eqref{funct-Jep}, then
$\uep$ are solutions to
\begin{equation}
\tag{$P_\ep(\fep, p_\ep)$}
\Delta_{p_\ep(x)}\uep={\beta}_{\varepsilon}(\uep)+\fep, \quad
u^{\ep}\geq 0
\end{equation}
and moreover, if the functions $\uep$, $\fep$ and $p_\ep$ are
uniformly bounded, we show that limit functions $u$ ($\ep\to 0$)
are solutions to the free boundary problem $P(f,p,{\lambda}^*)$
with $\lambda^*(x)=\Big(\frac{p(x)}{p(x)-1}\,M\Big)^{1/p(x)}$,
$M=\int \beta(s)\, ds$, $p=\lim p_\ep$, $f=\lim \fep$.

Problem $P_\ep(\fep,p_\ep)$, when $p_\ep(x)\equiv 2$ and $\fep\equiv 0$,
arises in combustion theory to
describe the propagation of curved premixed equi-diffusional
deflagration flames.  The study of the limit $(\ep\to 0)$ was
proposed in the 1930s and was first rigorously studied in \cite{BCN}.
The inhomogeneous case, $\fep\not\equiv 0$, allows the treatment
of more general combustion models with nonlocal diffusion and/or
transport.
In the case of the $p_\ep(x)$-Laplacian, this singular perturbation problem  may model flame propagation in a fluid with electromagnetic sensitivity.

Our work here, for both minimization problems, consists in an exhaustive analysis of the properties of nonnegative local minimizers, namely, global regularity and behavior close to the free boundary. This analysis allows us to prove that nonnegative local minimizers $u$ of \eqref{funct-J}, and functions $u=\lim \uep$ ($\ep\to 0$), with $\uep$ nonnegative local minimizers of \eqref{funct-Jep}, are weak solutions to the free boundary problem $P(f,p,{\lambda}^*)$ (Theorems \ref{AC=weak} and \ref{lim=weak}).

In order to obtain our results we need to overcome deep technical difficulties and develop new strategies, not present in the previous literature for this type of problems.

One of the results we would like to highlight is the proof of the Lipschitz continuity of nonnegative local minimizers of functional \eqref{funct-J} (Theorem \ref{pre-lip} and Corollary \ref{Lip}). Our proof relies on a careful rescaling argument, which transforms the problem into a minimization problem for a more general operator with nonstandard growth for which the control of the coefficients becomes nontrivial. This result, which is new for $f\not\equiv 0$, is also new in the homogeneous case $f\equiv 0$ for the range $1<p(x)<2$. It is worth remarking that minimization problems for the $p(x)$-Laplacian are of particular interest in the range $1<p(x)<2$ in the study of image processing (see \cite{AMS, CLR}). Therefore, we firmly believe that our estimates in Theorem \ref{pre-lip} are of independent interest.

Let us also emphasize that a key ingredient in many of our proofs is the use of rescaling arguments which, in particular, involve the handling of sequences of functions exhibiting  nonuniform integrability. Thus, the use of these kind of arguments for functional \eqref{funct-J} requires the introduction of the new concept of mild minimizers (see Definition  \ref{mild-minim}). Similar subtle ideas are also required when dealing with functional \eqref{funct-Jep} (see Theorems \ref{limit-minim} and \ref{densprop-sing}).

Once we achieve our goal, namely, once we prove  the  fundamental
properties of nonnegative local minimizers described above, we are able to apply results for solutions to the singular perturbation problem $P_\ep(\fep,p_\ep)$ and for weak solutions to the free boundary problem $P(f,p,{\lambda}^*)$ we recently obtained in our
 works \cite{LW4} and \cite{LW5}, respectively.

As a consequence we derive the smoothness of the free boundary for nonnegative local minimizers $u$ of \eqref{funct-J}. More precisely, we prove that
  the free boundary $\partial\{u>0\}$ is a
$C^{1,\alpha}$ surface with the exception of a subset of
${\mathcal H}^{N-1}$-measure zero (Theorem
\ref{regAC}).

In an analogous way, we get the smoothness of the free boundary for limit functions $u$ ($\ep\to 0$) of nonnegative local minimizers $\uep$ of \eqref{funct-Jep}, i.e., the free boundary $\partial\{u>0\}$ is a $C^{1,\alpha}$ surface with the exception
of a subset of ${\mathcal H}^{N-1}$-measure zero (Theorem  \ref{regminim}).

We also obtain further regularity results on the free boundary,
for both minimization problems, under further regularity
assumptions on the data (Corollaries \ref{high-AC} and \ref{high-persing}). In particular, if the data are analytic, the free
boundary is an analytic surface with the exception of a subset of
${\mathcal H}^{N-1}$-measure zero.

As stated above, the minimization problem with the functional in \eqref{funct-J} was first studied by Alt and
Caffarelli in \cite{AC} with $p(x)\equiv2$
and $f\equiv0$. Still in the homogeneous case $f\equiv0$, the problem was studied by Alt,
Caffarelli and Friedman in \cite{ACF} for a quasilinear
equation in the uniformly elliptic case, then the $p$-Laplacian ($p(x)\equiv p$) was treated in
\cite{DP}, an operator with power-like growth was studied in \cite{MW1}, and the case of a variable power $p(x)$ was
considered in \cite{FBMW}. The linear inhomogeneous case was
treated in \cite{GS} and \cite{Le}.

We remark  that the inhomogeneous minimization problem for functional \eqref{funct-J} with
$f\not\equiv0$ we consider here had not
been treated  in previous literature even in the case of
$p(x)\equiv p\neq 2$.

On the other hand, as pointed out above, problem $P_\ep(\fep,p_\ep)$ ---arising in combustion theory---
 was first rigorously studied in \cite{BCN} when $p_\ep(x)\equiv 2$ and $\fep\equiv 0$. Since then, much research has been
done on this problem, see \cite{CLW1, CLW2, CV, DPS, LO, LW1, MW, RT, W}.
For the inhomogeneous case we refer to \cite{LW2, LW3, MoWa1, MoWa2}.
Preliminary results for the $p_\ep(x)$-Laplacian were obtained in \cite{LW4}.

\medskip

We also remark  that the inhomogeneous minimization problem for functional \eqref{funct-Jep} with
 $\fep\not\equiv0$ we consider here had not
been treated  in previous literature even in the case of
$p_\ep(x)\equiv p_\ep\neq 2$. When $\fep\equiv 0$ our results are also new when $p_\ep(x)\not\equiv p_\ep$.

\medskip

An outline of the paper is as follows: In Section 2 we define the
notion of weak solution to the free boundary problem \ref{fbp-px}
and include some related definitions and results.
 In Section 3 we prove existence of minimizers of the energy
 functional \eqref{funct-J} and develop an exhaustive analysis of the essential   properties of functions $u$ which are nonnegative
 local minimizers of that energy.
In Section 4 we prove existence of minimizers of the energy
 functional \eqref{funct-Jep} and develop an analogous analysis of the properties of functions $\uep$  which are nonnegative
 local minimizers of that energy and moreover, we get results for their limit functions $u$. Finally, in Section 5 we study the regularity of the free boundary for both
minimization problems.
 We conclude the paper with an Appendix where we
 collect some results on variable exponent Sobolev spaces  as well as some other results that are used in the paper.

\bigskip

%%%%%%%%% preliminaries and notation %%%%%%%%%%%%%%%%%%%%%%%%%%%%%%%%%%%%%
\begin{subsection}{Preliminaries on Lebesgue and Sobolev spaces with variable
exponent}

Let $p :\Omega \to  [1,\infty)$ be a measurable bounded function,
called a variable exponent on $\Omega$ and denote $p_{\max} = {\rm
ess sup} \,p(x)$ and $p_{\min} = {\rm ess inf} \,p(x)$. We define
the variable exponent Lebesgue space $L^{p(\cdot)}(\Omega)$ to
consist of all measurable functions $u :\Omega \to \R$ for which
the modular $\varrho_{p(\cdot)}(u) = \int_{\Omega} |u(x)|^{p(x)}\,
dx$ is finite. We define the Luxemburg norm on this space by
$$
\|u\|_{L^{p(\cdot)}(\Omega)} = \|u\|_{p(\cdot)}  = \inf\{\lambda >
0: \varrho_{p(\cdot)}(u/\lambda)\leq 1 \}.
$$

This norm makes $L^{p(\cdot)}(\Omega)$ a Banach space.

There holds the following relation between $\varrho_{p(\cdot)}(u)$
and $\|u\|_{L^{p(\cdot)}}$:
\begin{align*}
\min\Big\{\Big(\int_{\Omega} |u|^{p(x)}\, dx\Big)
^{1/{p_{\min}}},& \Big(\int_{\Omega} |u|^{p(x)}\, dx\Big)
^{1/{p_{\max}}}\Big\}\le\|u\|_{L^{p(\cdot)}(\Omega)}\\
 &\leq  \max\Big\{\Big(\int_{\Omega} |u|^{p(x)}\, dx\Big)
^{1/{p_{\min}}}, \Big(\int_{\Omega} |u|^{p(x)}\, dx\Big)
^{1/{p_{\max}}}\Big\}.
\end{align*}

Moreover, the dual of $L^{p(\cdot)}(\Omega)$ is
$L^{p'(\cdot)}(\Omega)$ with $\frac{1}{p(x)}+\frac{1}{p'(x)}=1$.

Let $W^{1,p(\cdot)}(\Omega)$ denote the space of measurable
functions $u$ such that $u$ and the distributional derivative
$\nabla u$ are in $L^{p(\cdot)}(\Omega)$. The norm

$$
\|u\|_{1,p(\cdot)}:= \|u\|_{p(\cdot)} + \| |\nabla u|
\|_{p(\cdot)}
$$
makes $W^{1,p(\cdot)}(\Omega)$ a Banach space.

The space $W_0^{1,p(\cdot)}(\Omega)$ is defined as the closure of
the $C_0^{\infty}(\Omega)$ in $W^{1,p(\cdot)}(\Omega)$.

For the sake of completeness we include in an Appendix at the end of the paper some
additional results on these spaces that are used throughout the paper.
\end{subsection}

\bigskip

\begin{subsection}{Preliminaries on solutions to $p(x)$-Laplacian.} Let
$p(x)$ be as above, $g\in L^{\infty}(\Omega)$ and $a\in
L^{\infty}(\Omega)$, $a(x)\ge a_0>0$ in $\Omega$. We say that $u$
is a solution to
\begin{equation}\label{eq-con-a}
\mbox{div}(a(x)|\nabla u(x)|^{p(x)-2}\nabla u)= g(x) \ \mbox{ in }
\ \Omega
\end{equation}
if $u\in W^{1,p(\cdot)}(\Omega)$ and,  for every  $\varphi \in
C_0^{\infty}(\Omega)$, there holds that
$$
\int_{\Omega} a(x)|\nabla u(x)|^{p(x)-2}\nabla u \cdot \nabla
\varphi\, dx =-\int_{\Omega} \varphi\, g(x)\, dx.
$$
Under the assumptions of the present paper (see \ref{assump}
below) it follows as in Remark 3.2 in \cite{Wo} that $u\in
L_{\rm loc}^{\infty}(\Omega)$.

\bigskip

Moreover, for any $x\in\Omega$, $\xi,\eta\in\R^N$ fixed we have the
following inequalities
\begin{equation}
\label{desigualdades}
\begin{cases}
\ |\eta-\xi|^{p(x)}\leq C (|\eta|^{p(x)-2} \eta-|\xi|^{p(x)-2} \xi)\cdot
(\eta-\xi) & \quad  \mbox{ if } p(x)\geq 2,\\

\ |\eta-\xi|^2\Big(|\eta|+|\xi|\Big)^{p(x)-2}
\leq C (|\eta|^{p(x)-2} \eta-|\xi|^{p(x)-2} \xi) \cdot (\eta-\xi)& \quad
\mbox{ if } p(x)< 2,
\end{cases}
\end{equation}
with $C=C(N, p_{\min}, p_{\max})$. These inequalities imply that the function
$A(x,\xi)=a(x)|\xi|^{p(x)-2}\xi$ is strictly monotone. Then, the
comparison principle for equation \eqref{eq-con-a} holds on bounded domains since it
follows from the monotonicity of $A(x,\xi)$.

\end{subsection}

\begin{subsection}{Assumptions}\label{assump}

\smallskip

Throughout the paper we let $\Omega\subset\R^N$  be a domain.

\bigskip

\noindent{\bf Assumptions on $p_\ep(x)$ and $p(x)$.} We
assume that the functions $p_\ep(x)$ are measurable and verify
\begin{equation*}
1<p_{\min}\le p_\ep(x)\le p_{\max}<\infty,\qquad x\in\Omega.
\end{equation*}

For our main results we need  to assume further that $p_\ep(x)$ are uniformly
Lipschitz continuous in $\Omega$. In that case, we denote by $L$
the Lipschitz constant of $p_\ep(x)$, namely, $\|\nabla
p_\ep\|_{L^{\infty}(\Omega)}\leq L$.

Unless otherwise stated, the same assumptions above will be made
on the function $p(x)$.

When we are restricted to a ball $B_r$ we use $p_- = p_-(B_r)$ and
$p_+ = p_+(B_r)$ to denote the infimum and the supremum of $p(x)$
over $B_r$.

In some results we assume further that $p\in
W^{1,\infty}(\Omega)\cap W^{2,q}(\Omega)$, for some $q>1$.

\bigskip

\noindent{\bf Assumptions on $\lambda(x)$.} We assume that the
function $\lambda(x)$ is measurable in $\Omega$ and verifies
\begin{equation*}
0<\lone\le\lambda(x)\le\ltwo<\infty,\qquad x\in\Omega.
\end{equation*}

In some results we assume that $\lambda(x)$ is continuous in $\Omega$ and in our main results we assume further that $\lambda(x)$ is H\"older continuous in
$\Omega$.

\bigskip

\noindent{\bf Assumptions on $f_\ep(x)$ and $f(x)$.} We  assume that
$f_\ep, f\in L^{\infty}(\Omega)$. In some results we assume further that $f\in W^{1,q}(\Omega)$, for some $q>1$.

\bigskip

\noindent{\bf Assumptions on $\beta_\ep$.} We  assume that the
functions $\beta_\ep$ are defined by scaling of a single function
$\beta:\Bbb R\to \Bbb R$ satisfying:
\begin{itemize}
\item[i)] $\beta$ is a Lipschitz continuous function, \item [ii)]
$\beta>0$ in $(0,1)$ and $\beta\equiv 0$ otherwise, \item [iii)]
$\int_0^1\beta(s)\,ds=M$.
\end{itemize}
And then  $ \beta_\ep(s):=\frac{1}{\ep}\beta(\frac{s}{\ep}). $

\end{subsection}

\bigskip

\begin{subsection}{Notation}\ \ \newline

$\bullet$ $N$ \quad spatial dimension

$\bullet$   $\Omega\cap\partial\{ u>0 \}$ \quad free boundary

$\bullet$  $|S|$ \quad  $N$-dimensional Lebesgue measure of the
set $S$

$\bullet$ ${\mathcal H}^{N-1}$ \quad  $(N-1)$-dimensional
Hausdorff measure

$\bullet$  $B_r(x_0)$ \quad  open ball of radius $r$ and center
$x_0$

$\bullet$  $B_r$ \quad  open ball of radius $r$ and center $0$

$\bullet$  $B_r^+=B_r\cap\{x_N>0\}, \quad B_r^-=B_r\cap\{x_N<0\}$

$\bullet$  $B'_r(x_0)$ \quad  open ball of radius $r$ and center
$x_0$ in $\R^{N-1}$

$\bullet$  $B'_r$ \quad  open ball of radius $r$ and center $0$ in
$\R^{N-1}$

$\bullet$  $\fint_{B_r(x_0)}u= {1\over {|B_r(x_0)|}}
\int_{B_r(x_0)}u\,dx$

$\bullet$  $\fint_{\partial B_r(x_0)}u= {1\over {{\mathcal
H}^{N-1} (\partial B_r(x_0))}} \int_{\partial
B_r(x_0)}u\,d{\mathcal H}^{N-1}$

$\bullet$ $\chi_{{}_S}$ \quad  characteristic function of the set
$S$

$\bullet$ $u^{+}=\text{\rm max}(u,0)$,\quad $u^{-}=\text{\rm
max}(-u,0)$

$\bullet$ $\langle\,\xi\, ,\,\eta\,\rangle$ \, and \, $\xi \cdot
\eta$ \quad both denote scalar product in $\Bbb R^{N}$

$\bullet$ $B_\ep(s)=\int _0^s\beta_\ep(\tau) \, d\tau$

%%%%%%%%%% end of preliminaries and notation %%%%%%%%%%%%%%%%%%%%%%%%%%%%%%%%
\end{subsection}

\end{section}
%%%%%%%%%%  end of introduction %%%%%%%%%%%%%%%%%%%%%%%%%%%%%%%%%%%%%

%%%%%%%%%%%%%section weak solutions to the free boundary problem%%%%%%%%%%%%%%%%%%%

\begin{section}{Weak solutions to the free boundary problem $P(f,p,{\lambda}^*)$}
\label{sect-weak-solut} In this section, for the sake of
completeness, we define the notion of weak solution to the free
boundary problem \ref{fbp-px} and we give other related
definitions and results that we are going to employ in the paper.

We point out that in \cite{LW5} we derived some properties of the
weak solutions to problem \ref{fbp-px} and we developed a theory
for the regularity of the free boundary for weak solutions.

In this section $p(x)$ will be a Lipschitz continuous function.

We first need

\begin{defi}
\label{clw2-Definition 3.1} Let $u$ be a continuous and
nonnegative function in a domain $\Omega\subset \Bbb R^{N}$. We
say that $\nu$ is the exterior unit normal to the free boundary
$\Omega\cap\fb$ at a point $x_0\in\Omega\cap\fb$ in the
 measure theoretic sense, if $\nu
\in\Bbb R^N$, $|\nu|=1$ and
\begin{equation*}
 \lim_{r\to 0} \frac1{r^{N}} \int_{B_r(x_0)}
|\chi_{\{u>0\}}- \chi_{\{x\,/\, \langle x-x_0,\nu\rangle
<0\}}|\,dx = 0.
\end{equation*}
\end{defi}

Then we have

\begin{defi}\label{weak2} Let $\Omega\subset \Bbb R^{N}$ be a domain. Let $p$ be a
measurable function in $\Omega$ with $1<p_{\min}\le p(x)\le
p_{\max}<\infty$, $\lstar$ continuous in $\Omega$ with
$0<\lone\le\lstar(x)\le\ltwo<\infty$ and $f\in L^\infty(\Omega)$.
We call $u$ a weak solution of \ref{fbp-px} in $\Omega$ if
\begin{enumerate}
\item $u$ is continuous and nonnegative in $\Omega$, $u\in
W_{\rm loc}^{1,p(\cdot)}(\Omega)$ and $\Delta_{p(x)}u=f$ in
$\Omega\cap\{u>0\}$. \item For $D\subset\subset \Omega$ there are
constants $c_{\min}=c_{\min}(D)$, $C_{\max}=C_{\max}(D)$,
$r_0=r_0(D)$, $0< c_{\min}\leq C_{\max}$, $r_0>0$, such that for
balls $B_r(x)\subset D$ with $x\in
\partial \{u>0\}$ and $0<r\le r_0$
$$
c_{\min}\leq \frac{1}{r}\sup_{B_r(x)} u \leq C_{\max}.
$$
\item For $\mathcal{H}^{N-1}$ a.e.
$x_0\in\partial_{\rm{red}}\{u>0\}$ (that is, for ${\mathcal
H}^{N-1}$-almost every point $x_0\in\Omega\cap\fb$ such that
$\Omega\cap\partial\{u> 0\}$ has an exterior  unit normal
 $\nu(x_0)$ in the measure theoretic sense)
$u$ has the asymptotic development
\begin{equation*}
u(x)=\lambda^*(x_0)\langle x-x_0,\nu(x_0)\rangle^-+o(|x-x_0|).
\end{equation*}

\item For every $ x_0\in \Omega\cap\partial\{u>0\}$,
\begin{align*}
& \limsup_{\stackrel{x\to x_0}{u(x)>0}} |\nabla u(x)| \leq
\lambda^*(x_0).
\end{align*}

If there is a ball $B\subset\{u=0\}$ touching
$\Omega\cap\partial\{u>0\}$ at $x_0$, then
$$\limsup_{\stackrel{x\to x_0}{u(x)>0}} \frac{u(x)}{\mbox{dist}(x,B)}\geq  \lambda^*(x_0). $$
\end{enumerate}
\end{defi}

\begin{defi}\label{nondegener} Let $v$ be a continuous nonnegative function
in a domain $\Omega\subset\mathbb{R}^N$. We say that $v$ is
nondegenerate at a point  $x_0\in \Omega\cap\{v=0\}$ if there
exist $c>0$, $\bar r_0>0$ such that one of the following
conditions holds:
\begin{equation}\label{nond-prom-bol}
\fint_{B_r(x_0)} v\, dx\geq c r\quad \mbox{ for } 0<r\leq \bar
r_0,
\end{equation}
\begin{equation}\label{nond-prom-casc}
\fint_{\partial B_r(x_0)} v\, dx\geq c r\quad \mbox{ for } 0<r\leq
\bar r_0,
\end{equation}
\begin{equation}\label{nond-sup}
\sup_{B_r(x_0)} v\geq c r\quad \mbox{ for } 0<r\leq \bar r_0.
\end{equation}

\medskip

We say that $v$ is uniformly nondegenerate on a set
$\Gamma\subset\Omega\cap\{v=0\}$ in the sense of
\eqref{nond-prom-bol} (resp. \eqref{nond-prom-casc},
\eqref{nond-sup}) if the constants $c$ and $\bar r_0$ in
\eqref{nond-prom-bol} (resp.  \eqref{nond-prom-casc},
\eqref{nond-sup}) can be taken independent of the point
$x_0\in\Gamma$.
\end{defi}

\begin{rema}\label{equiv-nondeg}
Assume that $v\ge 0$ is locally Lipschitz continuous in a domain
$\Omega\subset\mathbb{R}^N$, $v\in W^{1,p(\cdot)}(\Omega)$ with
$\Delta_{p(x)} v \ge  f \chi_{\{v >0\}}$, where $f\in
L^{\infty}(\Omega)$, $1<p_{\min}\le p(x)\le p_{\max}<\infty$ and
$p(x)$ is Lipschitz continuous. Then the three concepts of
nondegeneracy in Definition \ref{nondegener} are equivalent (for
the idea of the proof, see Remark 3.1 in \cite{LW1}, where the
case $p(x)\equiv 2$ and $f\equiv 0$ is treated).
\end{rema}

\end{section}

%%%%%%%%%%end section weak solutions %%%%%%%%%%%%%%%%%%

%%%%%%%%%%% section energy minimizers AC %%%%%%%%%%%%%%%%%%%%%%%%%%%%%%%%%%%%%%%%
\begin{section}{Energy minimizers of energy functional \eqref{funct-J}}
\label{sect-energy-minim-ac}

\smallskip
In this section  we prove existence of minimizers of the energy functional \eqref{funct-J} and we develop an exhaustive analysis of the essential   properties
of functions $u$ which
are nonnegative local  minimizers of that energy.

We start with a definition and some related remarks

\begin{defi}\label{def-loc-min} Let $1<p_{\min}\le p(x)\le p_{\max}<\infty$, $f\in L^{\infty}(\Omega)$ and $\lambda(x)$ measurable with
$0<\lone\le\lambda(x)\le\ltwo<\infty$.
We say that $u\in  W^{1,p(\cdot)}(\Omega)$ is a local minimizer in $\Omega$ of
\begin{equation*}
 J(v)=J_{\Omega}(v)=\int_\Omega\Big(\frac{|\nabla
 v|^{p(x)}}{p(x)}+\lambda(x)\chi_{\{v>0\}}+fv\Big)\,dx
\end{equation*}
if for every $\Omega'\subset\subset\Omega$ and for every $v\in W^{1,p(\cdot)}(\Omega)$ such that $v=u$ in
$\Omega\setminus{\Omega'}$ there holds that $J(v)\ge J(u)$.
\end{defi}

\begin{rema}\label{rem-local} Let $u$ be as in Definition \ref{def-loc-min}. Let $\Omega'\subset\subset\Omega$ and $w-u\in W_0^{1,p(\cdot)}(\Omega')$. If we define
\begin{equation*}
\bar w=
\begin{cases}
w  & \mbox{in }\Omega',\\
u & \mbox{in }\Omega\setminus{\Omega'},
\end{cases}
\end{equation*}
then  $\bar w\in W^{1,p(\cdot)}(\Omega)$ and therefore $J(\bar w)\ge J(u)$. If we now let
\begin{equation*}
 J_{\Omega'}(v)=\int_{\Omega'}\Big(\frac{|\nabla
 v|^{p(x)}}{p(x)}+\lambda(x)\chi_{\{v>0\}}+fv\Big)\,dx,
\end{equation*}
it follows that $J_{\Omega'}(w)\ge J_{\Omega'}(u)$.
\end{rema}

\begin{rema}\label{rem-global} Let $J$ be as in Definition \ref{def-loc-min}. If $u\in W^{1,p(\cdot)}(\Omega)$ is a minimizer of $J$ among the functions
$v\in u+W_0^{1,p(\cdot)}(\Omega)$, then $u$ is a local minimizer of $J$ in $\Omega$.
\end{rema}

We first prove

\begin{theo}\label{existence-minimizers-AC} Assume that $1<p_{\min}\le p(x)\le p_{\max}<\infty$ with $\|\nabla
p\|_{L^{\infty}}\leq L$, $f\in L^{\infty}(\Omega)$ and $\lambda(x)$ is measurable with
$0<\lone\le\lambda(x)\le\ltwo<\infty$. Let $\phi\in W^{1,p(\cdot)}(\Omega)$ and assume that $\Omega$ is a bounded domain.
There exists
$u\in W^{1,p(\cdot)}(\Omega)$ that minimizes the energy
\begin{equation*}
J(v)=\di \int_\Omega \Big(\frac{|\nabla
 v|^{p(x)}}{p(x)}+\lambda(x)\chi_{\{v>0\}}+fv\Big)\,dx,
\end{equation*}
among functions $v\in W^{1,p(\cdot)}(\Omega)$ such that $v-\phi \in W_0^{1,p(\cdot)}(\Omega)$.
Then,  for every $\Omega'\subset\subset\Omega$ there exists
$C=C(\Omega', \|\phi\|_{1,p(\cdot)}, \|f\|_{L^\infty(\Omega)}, p_{\min}, p_{\max},\ltwo, L)$ such that
\begin{equation}\label{cotsup-AC}
\sup_{\Omega'} u\le C.
\end{equation}
\end{theo}
\begin{proof} Let us prove first that a minimizer exists. In fact, let
$$
\mathcal{K}=\Big\{v\in W^{1,p(\cdot)}(\Omega)\colon v - \phi\in W_0^{1,p(\cdot)}(\Omega)\Big\}.
$$
In order to prove that $J$ is bounded from below in $\mathcal{K}$,  we observe that if $v\in\mathcal{K}$, then
$$
J(v)\ge\frac{1}{p_{\max}}\int_\Omega |\nabla v|^{p(x)}\, dx+ \int_\Omega f v\, dx,
$$
and we have, by Theorem \ref{holder} and Theorem \ref{poinc},
\begin{align*}
\int_\Omega |f v|\, dx &\le 2\|f\|_{{p}'(\cdot)}\|v\|_{p(\cdot)}\le 2\|f\|_{{p}'(\cdot)}(\|v-\phi\|_{p(\cdot)}
+\|\phi\|_{p(\cdot)})\\
&\le C_0\|\nabla v-\nabla\phi\|_{p(\cdot)}+C_1\le  C_0\|\nabla
v\|_{p(\cdot)}+C_2.
\end{align*}
If  $\Big(\int_{\Omega} |\nabla v|^{p(x)}\, dx\Big)^{1/{p_{\min}}}\ge \Big(\int_{\Omega} |\nabla v|^{p(x)}\, dx\Big)^{1/{p_{\max}}}$ we get,
by Proposition \ref{equi},
$$
\int_\Omega |f v|\, dx \le  C_0\Big(\int_{\Omega} |\nabla
v|^{p(x)}\, dx\Big)^{1/{p_{\min}}}+C_2\le C_3 + \frac{1}{2\,
p_{\max}}\int_\Omega |\nabla v|^{p(x)}\, dx.
$$
If, on the other hand, $\Big(\int_{\Omega} |\nabla v|^{p(x)}\,
dx\Big)^{1/{p_{\min}}}< \Big(\int_{\Omega} |\nabla v|^{p(x)}\,
dx\Big)^{1/{p_{\max}}}$, we get in an analogous way
$$
\int_\Omega |f v|\, dx \le  C_0\Big(\int_{\Omega} |\nabla
v|^{p(x)}\, dx\Big)^{1/{p_{\max}}}+C_2\le  C_4 + \frac{1}{2\,
p_{\max}}\int_\Omega |\nabla v|^{p(x)}\, dx.
$$
Taking $C_5=\max\{C_3, C_4\}$, we get
\begin{equation}\label{cotaJ-AC}
J(v)\ge -C_5 + \frac{1}{2\, p_{\max}}\int_\Omega |\nabla
v|^{p(x)}\, dx\ge -C_5,
\end{equation}
which shows that $J$ is bounded from below in $\mathcal{K}$.

At this point we want to remark that the constants $C_0,...,C_5$ above can be taken depending only on
$\|\phi\|_{1,p(\cdot)}, \|f\|_{L^\infty(\Omega)}, p_{\min}, p_{\max}$ and $L$.

We now take  a minimizing sequence $\{u_n\}\subset\K$. Without loss of generality we can assume that $J(u_n)\le J(\phi)$, so
by  \eqref{cotaJ-AC},$\int_{\Omega}|\nabla u_n|^{p(x)}\le C_6$. By Proposition \ref{equi}, $\|\nabla
u_n-\nabla \phi\|_{p(\cdot)}\leq C_7$ and, as  $u_n - \phi\in W_0^{1,p(\cdot)}(\Omega)$, by Theorem \ref{poinc} we
 have $\|u_n-\phi\|_{p(\cdot)}\leq C_8$. Therefore, by Theorem
\ref{ref}
 there exist a
subsequence (that we still call $u_n$) and a function $u\in
W^{1,p(\cdot)}(\Omega)$ such that
\begin{equation}
\label{cotanorma-AC} ||u||_{W^{1,p(\cdot)}(\Omega)}\le {\bar C},\quad \mbox{ with } {\bar C}={\bar C}(\|\phi\|_{1,p(\cdot)},
\|f\|_{L^\infty(\Omega)}, p_{\min}, p_{\max},\ltwo, L),
\end{equation}
$$
u_n \rightharpoonup  u\quad \mbox{weakly in } W^{1,p(\cdot)}(\Omega),
$$
and, by Theorem \ref{imb},
\begin{align*}
&  u_n \rightharpoonup  u \quad \mbox{weakly in }
W^{1,p_{\min}}(\Omega).
\end{align*}
Now, by the compactness of the immersion
$W^{1,{p_{\min}}}(\Omega)\hookrightarrow L^{p_{\min}}(\Omega)$ we
have that, for a subsequence that we still denote by $u_n$,
\begin{align*}
u_n &\to u \quad \mbox{in }L^{p_{\min}}(\Omega),\\
u_n &\to u \quad \mbox{a.e. } \Omega.
\end{align*}

As $\K$ is convex and closed, it is weakly closed, so $u\in \K$.

It follows that
\begin{align*}
\lambda(x)\chi_{\{u>0\}}&\le\liminf_{n\to\infty} \lambda(x)\chi_{\{u_n>0\}},\\
\int_{\Omega}\lambda(x)\chi_{\{u>0\}}\, dx &\le \liminf_{n\to\infty} \int_{\Omega} \lambda(x)\chi_{\{u_n>0\}}\, dx, \\
\lim_{n\to\infty} \int_{\Omega} f u_n\, dx &=  \int_{\Omega} f u\, dx,\\
 \int_{\Omega} \frac{|\nabla u|^{p(x)}}{p(x)}\, dx &\le \liminf_{n\to\infty}
\int_{\Omega}\frac{|\nabla u_n|^{p(x)}}{p(x)}\, dx.
\end{align*}
In order to prove the last inequality we observe that there holds
\begin{equation}\label{G-AC}
\int_\Omega \frac{|\nabla u_n|^{p(x)}}{p(x)}\,dx\ge\int_\Omega \frac{|\nabla u|^{p(x)}}{p(x)}\,dx
+\int_\Omega |\nabla u|^{p(x)-2}\nabla u\cdot(\nabla u_n-\nabla u)\,dx.
\end{equation}

Recall that $\nabla u_n$ converges weakly to $\nabla u$ in
$L^{p(\cdot)}(\Omega)$. Now, since  $|\nabla u|^{p(x)-1}\in L^{p'(\cdot)}(\Omega)$, by
Theorem \ref{ref} and passing to the limit in \eqref{G-AC} we get
$$
\liminf_{n\to\infty}\int_\Omega \frac{|\nabla u_n|^{p(x)}}{p(x)}\,dx\ge
\int_\Omega \frac{|\nabla u|^{p(x)}}{p(x)}\,dx.
$$

Hence
$$
J(u)\le \liminf_{n\to\infty}J(u_n) =
\inf_{v\in\K} J(v).
$$
Therefore, $u$ is a minimizer of $J$ in $\K$.

Finally, in order to prove \eqref{cotsup-AC}, we observe that, from Proposition \ref{equi} and estimate \eqref{cotanorma-AC},
we have that $\int_{\Omega}|u|^{p(x)}\, dx \le {\bar C_1}(\|\phi\|_{1,p(\cdot)}, \|f\|_{L^\infty(\Omega)}, p_{\min}, p_{\max},\ltwo, L)$. Thus, the desired estimate
follows from the application of Proposition 2.1 in \cite{Wo}, since, by Lemma \ref{lemm-min-desig}, $\Delta_{p(x)} u\ge f \ge -\|f\|_{L^\infty(\Omega)}$ in $\Omega$.
\end{proof}

{}For local minimizers we first have

\begin{lemm}\label{lemm-min-desig} Let $p, f$ and $\lambda$ be as in Theorem \ref{existence-minimizers-AC}.
 Let $u\in  W^{1,p(\cdot)}(\Omega)$  be a  local minimizer of
\begin{equation*}
 J(v)=\int_\Omega\Big(\frac{|\nabla
 v|^{p(x)}}{p(x)}+\lambda(x)\chi_{\{v>0\}}+fv\Big)\,dx.
\end{equation*}
 Then
 \begin{equation}\label{min-desig}
\Delta_{p(x)} u  \ge f \quad \mbox{ in }\Omega.
\end{equation}
\end{lemm}
\begin{proof}
In fact, let $t>0$ and $0\le\xi\in C^{\infty}_0(\Omega)$. Using
the minimality of $u$  we have
\begin{align*}0 &\leq \frac{1}{t} (J(u-t
\xi)-J(u)) \le \frac{1}{t} \int_{\Omega}\Big(\frac{|\nabla u-t
\nabla \xi|^{p(x)}}{p(x)}-\frac{|\nabla u|^{p(x)}}{p(x)}\Big)\, dx
\,
-\int_{\Omega}f\xi\, dx\\
& \leq -\int_{\Omega}
 |\nabla u-t \nabla \xi|^{p(x)-2}(\nabla u-t \nabla
\xi)\cdot \nabla \xi\, dx -\int_{\Omega}f\xi\, dx
\end{align*}
and if we take $t\rightarrow 0$, we obtain
\begin{equation}\label{desig}
0\leq -\int_{\Omega} |\nabla u|^{p(x)-2}\nabla u\cdot\nabla \xi\,
dx-\int_{\Omega}f\xi\, dx,
\end{equation}
which gives \eqref{min-desig}.
 \end{proof}

\begin{rema}\label{rema-signo-AC} We are interested in studying the behavior of nonnegative local minimizers of the energy functional
\eqref{funct-J}.

If $u$ is as in Theorem \ref{existence-minimizers-AC} and we have, for instance, $\phi\ge 0$ in $\Omega$  and $f\le 0$ in $\Omega$, then we have $u\ge 0$ in $\Omega$. In fact, the result follows by observing that $\xi=\min (u,0)\in W_0^{1,p(\cdot)}(\Omega)$ so, for every $0<t<1$,
$u-t\xi\in\phi + W_0^{1,p(\cdot)}(\Omega)$, with $\chi_{\{u-t\xi>0\}}=\chi_{\{u>0\}}$. Then, in a similar way as in
Lemma \ref{lemm-min-desig}, we get \eqref{desig} and using that $f\le 0$ we obtain $\int_{\Omega} |\nabla \xi|^{p(x)}\,dx=0$, which implies $u\ge 0$ in $\Omega$.

On the other hand, if $u$ is any local minimizer of \eqref{funct-J}, the same argument employed in Theorem \ref{existence-minimizers-AC} gives
$\sup_{\Omega'} u\le C_{\Omega'}$, for any $\Omega'\subset\subset\Omega$. Therefore, if $u$ is any nonnegative local minimizer of \eqref{funct-J}, then $u\in L_{\rm loc}^{\infty}(\Omega)$.
\end{rema}

\medskip

{}From now on we will deal with nonnegative local minimizers. Next we will prove that they are locally
Lipschitz continuous.

First we need

\begin{lemm}\label{ppio-max-con-a} Let $p$ and  $f$ be as in Theorem \ref{existence-minimizers-AC}.
Let $\Omega\subset (0,d)\times\mathbb{R}^{N-1}$ be a bounded
domain. Assume
$a\in L^{\infty}(\Omega)$, $a(x)\ge a_0>0$, with $\|\nabla
a\|_{L^{\infty}}\leq L_1$.
 Let $u\in  W^{1,p(\cdot)}(\Omega)$  be a solution to ${\rm
div}\big(a(x)|\nabla u|^{{ p}(x)-2}\nabla u\big)={f}$ in $\Omega$
with $|u|\le M$ on $\partial\Omega$. Assume moreover that $Ld < p_{\min}-1$.

Then, there exists $C=C(M,p_{\min},
||f||_{L^{\infty}(\Omega)}, d, a_0, L, L_1)$ such  that $|u|\le C$ in $\Omega$.
\end{lemm}
\begin{proof}
We consider, for $\alpha>1$, the function $w(x)=M+e^{\alpha d}-e^{\alpha x_1}$.
Computing, we have
\begin{align*}
w_{x_i}=-\alpha e^{\alpha x_1}\delta_{i1},\quad  w_{x_i x_j}= -{\alpha}^2 e^{\alpha x_1}\delta_{i1}\delta_{j1},\quad  |\nabla w|=\alpha e^{\alpha x_1}.
\end{align*}
Therefore  we obtain
\begin{align*} {\rm div}&\big(a(x)|\nabla w|^{{ p}(x)-2}\nabla w\big)\\
&=|\nabla w|^{p(x)-2}
\Big[a(x)\Delta w+a(x)\langle \nabla w,\nabla
p\rangle \log |\nabla w|+a(x)\frac{(p(x)-2)}{|\nabla w|^2}\sum_{i,j}
w_{x_i}w_{x_j}w_{x_ix_j}+ \langle \nabla w,\nabla
a\rangle \Big]\\
&=a(x)(\alpha e^{\alpha x_1})^{p(x)-1}\Big[-(p(x)-1)\alpha-p_{x_1}(x)\log(\alpha e^{\alpha x_1})-\frac{a_{x_1}(x)}{a(x)}\Big]\\
&\le a(x)(\alpha e^{\alpha x_1})^{p(x)-1}\Big[-(p_{\min}-1)\alpha + L\log\alpha +L\alpha x_1 + \frac{|a_{x_1}(x)|}{a(x)}\Big]\\
&\le a(x)(\alpha e^{\alpha x_1})^{p(x)-1}\Big[\big(-(p_{\min}-1) + Ld \big)\alpha+ L\log\alpha + \frac{L_1}{a_0}\Big].
\end{align*}
If we let $\alpha\ge\alpha_0=\alpha_0(p_{\min},d, a_0, L, L_1)$ so that $\big(-(p_{\min}-1) + Ld \big)\alpha+ L\log\alpha + \frac{L_1}{a_0}<0$, we get
\begin{align*} {\rm div}&\big(a(x)|\nabla w|^{{ p}(x)-2}\nabla w\big)\\
&\le a_0 {\alpha}^{p_{\min}-1}\Big[\big(-(p_{\min}-1) + Ld \big)\alpha+ L\log\alpha  + \frac{L_1}{a_0}\Big]\\
&\le -||f||_{L^{\infty}(\Omega)},
\end{align*}
where the last inequality holds if we choose $\alpha\ge\alpha_1=\alpha_1(||f||_{L^{\infty}(\Omega)},p_{\min},d, a_0, L, L_1)$.

It follows that for $\alpha=\max\{\alpha_0, \alpha_1, 1\}$ the corresponding function $w$ satisfies
$$
{\rm div}\big(a(x)|\nabla w|^{{ p}(x)-2}\nabla w\big)\le -||f||_{L^{\infty}(\Omega)}\le \pm f\quad{\rm in}\quad\Omega.
$$
Since $\pm u\le w$ on $\partial\Omega$, we get $\pm u\le w\le M+
e^{\alpha d}$ in $\Omega$. This concludes the proof.
\end{proof}

\begin{rema}\label{ppio-max-gral} Let $u$ be as in Lemma \ref{ppio-max-con-a} in a domain $\Omega\subset (-r,r)\times\mathbb{R}^{N-1}$.
Then, defining $\bar u(x)=u(x-re_1)$, $\bar a(x)=a(x-re_1)$, $\bar p(x)=p(x-re_1)$, $\bar f(x)=f(x-re_1)$ and $\bar \Omega=\Omega+re_1$, we have
${\rm
div}\big(\bar a(x)|\nabla \bar u|^{{ \bar p}(x)-2}\nabla \bar u\big)={\bar f}$ in $\bar\Omega$. Then, the invariance by translations of the problem allows us to apply Lemma \ref{ppio-max-con-a}  to $\bar u$ and conclude that, if  $L2r < p_{\min}-1$, then $|u|\le C$ in $\Omega$, for a constant
$C=C(M,p_{\min},||f||_{L^{\infty}(\Omega)}, r, a_0, L, L_1)$.
\end{rema}

Next, we prove that nonnegative local minimizers ---of a more general functional than \eqref{funct-J}---   are
locally H\"older continuous.

\begin{theo}\label{teo-holder} Let $p, f$ and $\lambda$ be as in Theorem \ref{existence-minimizers-AC}.
Assume that  $0<a_0\le a(x)\le a_1<\infty$, with $\|\nabla a\|_{L^{\infty}}\leq L_1$.
 Let  $u\in W^{1,p(\cdot)}(\Omega)\cap L^{\infty}(\Omega)$  be a nonnegative local
minimizer of
\begin{equation*}
 J^{a}(v)=\int_\Omega\Big(a(x)\frac{|\nabla
 v|^{p(x)}}{p(x)}+\lambda(x)\chi_{\{v>0\}}+fv\Big)\,dx
\end{equation*}
and let $B_{\hat r_0}(x_0)\subset \Omega$. Then, there exist
$0<\gamma<1$ and $0<\hat\rho_0< \hat r_0$,
$\hat\rho_0=\hat\rho_0(\hat r_0, N, p_{\min}, L)$ and
$\gamma=\gamma( N, p_{\min})$,   such that $u\in
C^{\gamma}(\overline{B_{\hat\rho_0}(x_0)})$. Moreover,
$\|u\|_{C^{\gamma}(\overline{B_{\hat\rho_0}(x_0)})}\leq C$ with
$C$ depending only on $N$, $\hat r_0$, $p_{\min}$, $p_{\max}$,
$L$, $\ltwo$, $\|u\|_{L^{\infty}(B_{\hat r_0}(x_0))}$,
$\|f\|_{L^{\infty}(B_{\hat r_0}(x_0))}$, $a_0$, $a_1$ and $L_1$.
\end{theo}
\begin{proof}
We will prove that  there exist $0<\gamma<1$  and $0<\rho_0< r_0<
\hat r_0$ such that, if $B_{r_0}(y)\subset B_{\hat r_0}(x_0)$ and
$\rho\leq \rho_0$,   then
\begin{equation}\label{algosale} \Big(\pint_{B_\rho(y)}|\nabla
u|^{p_-}\,dx\Big)^{1/{p_-}}\le C \rho^{\gamma-1},
\end{equation}
where $p_-=p_-(B_{r_0}(y))$. Without loss of generality we will
assume that $y=0$.

In fact, let  $0<r_0\le \min\{\frac{\hat r_0}{2}, 1\}$, $0<r\leq
r_0$ and $v$ the solution of
\begin{equation}\label{ecu-v}
{\rm div}\big(a(x)|\nabla v|^{{ p}(x)-2}\nabla v\big)=f \quad\mbox{in }B_r, \qquad v-u\in
W_0^{1,p(\cdot)}(B_r).\end{equation}

If $r_0\le \frac{1}{4L}(p_{\min}-1)$, it follows from Lemma \ref{ppio-max-con-a} and Remark \ref{ppio-max-gral} that
\begin{equation}\label{cota-v}
||v||_{L^{\infty}(B_{r})}\le \bar C\ \quad \mbox{ with } \quad\bar
C=\bar C(L, p_{\min},  \|u\|_{L^{\infty}(B_{\hat
r_0}(x_0))}, \|f\|_{L^{\infty}(B_{\hat r_0}(x_0))}, a_0, L_1).
\end{equation}

Let $u^s(x)=s u(x)+(1-s) v(x)$. By using \eqref{ecu-v} and the
inequalities in \eqref{desigualdades}, we get
\begin{equation}\label{standard-hold}\begin{aligned}
&\int_{B_{r}} a(x)\frac{|\nabla u|^{p(x)}}{p(x)}- a(x)\frac{|\nabla v|^{p(x)}}{p(x)}+ \int_{B_{r}}{f}({u}-v)=\\
&\quad\int_0^1 \frac{ds}{s}\int_{B_{r}}a(x)\Big(|\nabla u^s|^{{{
p}(x)}-2}\nabla u^s-
|\nabla v|^{{{p}(x)}-2}\nabla v\Big) \cdot \nabla(u^s-v)\ge\\
&\quad\quad C\Big(\int_{B_{r}\cap\{{{p}}\geq 2\} }a(x) |\nabla {
u}-\nabla v|^{{p}(x)} +\int_{B_{r}\cap\{{{p}}<2\} }a(x) |\nabla {
u}-\nabla v|^2\Big(|\nabla {u}|+|\nabla v|\Big)^{{{
p}(x)}-2}\Big),
\end{aligned}
\end{equation}
where $C=C(p_{\min},p_{\max},N)$.

Therefore,  by the minimality of $u$, we have (if
$A_1=B_r\cap\{p(x)<2\}$ and $A_2=B_r\cap\{p(x)\geq 2\}$)
\begin{align}\label{rn2}&\int_{ A_2} |\nabla u-\nabla v|^{p(x)}  \, dx\leq
C r^N,\\\label{rn3} &\int_{A_1}|\nabla u-\nabla v|^2(|\nabla
u|+|\nabla v|)^{p(x)-2} \, dx\leq C r^N,
\end{align}
where $C=C(p_{\min},p_{\max},N, \ltwo, a_0)$.

 Let $\ep>0$. Take
$\rho=r^{1+\ep}$  and suppose that $r^{\ep}\leq 1/2$. Take
$0<\eta<1$ to be chosen later. Then, by Young's inequality, the
definition of $A_1$ and \eqref{rn3}, we obtain
\begin{equation}\label{rn4}
\begin{aligned}
\int_{A_1\cap B_{\rho}} |\nabla u-\nabla v|^{p(x)}\, dx \leq&
\frac{C}{{\eta}^{2/{p_{\min}}}}\int_{A_1\cap B_{r}}(|\nabla
u|+|\nabla v|)^{p(x)-2}|\nabla u-\nabla v|^2 \,
dx\\
&+ C\eta\int_{ B_{\rho}\cap A_1} (|\nabla u|+|\nabla v|)^{p(x)}\,
dx\\
\leq& \frac{C}{{\eta}^{2/{p_{\min}}}}  r^{N}+ C\eta\int_{
B_{\rho}\cap A_1} (|\nabla u|+|\nabla v|)^{p(x)}\, dx.
\end{aligned}
\end{equation}
Therefore, by \eqref{rn2} and \eqref{rn4}, we get
\begin{equation}\label{rn51}\int_{B_{\rho}}|\nabla u-\nabla v|^{p(x)}\, dx
 \leq \frac{C}{{\eta}^{2/{p_{\min}}}} r^N+C\eta\int_{B_{\rho}\cap A_1}(|\nabla u|+|\nabla v|)^{p(x)}\, dx,\end{equation}
where $C=C(p_{\min},p_{\max},N, \ltwo, a_0)$.

Since, $|\nabla u|^q\leq C(|\nabla u-\nabla v|^q+|\nabla v|)^q)$,
for any $q>1$, with $C=C(q)$, we have, by \eqref{rn51}, choosing
$\eta$ small, that
\begin{equation}\label{otra}\int_{B_{\rho}} |\nabla u|^{p(x)}\, dx
 \leq {C} r^N+C\int_{B_{\rho}} |\nabla v|^{p(x)}\, dx,\end{equation}
where $C=C(p_{\min},p_{\max},N, \ltwo, a_0)$.

Now let $M\ge 1$ such that $||v||_{L^{\infty}(B_{r})}\le M$ and
define
$$
w(x) =\frac{v(rx)}{M} \quad {\rm in} \quad B_1.
$$
Then, there holds that
\begin{equation*}
{\rm div}\big(\bar a(x)|\nabla w|^{{ \bar p}(x)-2}\nabla w\big)
={\left(\frac{r}{M}\right)}^{p(rx)-1}rf(rx)+r\log\left(\frac{r}{M}\right)a(rx)\nabla p(rx)\cdot \nabla w(x)|\nabla w(x)|^{p(rx)-2}
\end{equation*}
in $B_1$, with $\bar p(x)=p(rx)$ and $\bar a(x)=a(rx)$. That is,
\begin{equation*}
{\rm div}\big(\bar a(x)|\nabla w|^{{ \bar p}(x)-2}\nabla w\big)=B(x,\nabla w(x))
\quad\mbox{in }B_1,
\end{equation*}
with
\begin{equation*}
|B(x,\nabla w(x))|\le C \left(1+|\nabla w(x)|^{\bar p(x)}\right)
\quad\mbox{in }B_1,
\end{equation*}
where  $C=C(L, M, \|f\|_{L^{\infty}(B_{\hat r_0}(x_0))}, a_1)$.

{ }From Theorem 1.1 in \cite{Fan}, it follows that $w\in
C^{1,\alpha}_{\rm loc}(B_1)$ for some $0<\alpha<1$ and that
\begin{equation*}
\sup_{B_{1/2}} |\nabla w|\le C(L, M, \|f\|_{L^{\infty}(B_{\hat r_0}(x_0))},p_{\min},p_{\max}, N, a_0, a_1, L_1),
\end{equation*}
which implies
\begin{equation}\label{cota-gv}
\sup_{B_{r/2}} |\nabla v|\le \frac{CM}{r}.
\end{equation}

Therefore, from \eqref{otra} and \eqref{cota-gv}, we deduce that
\begin{equation}\label{otra2}
\int_{B_{\rho}} |\nabla u|^{p(x)}\, dx
 \leq C r^N + C{\rho}^N r^{-p_+},
\end{equation}
with $p_+=p_+(B_{r_0})$ and  $C=C(L, \|u\|_{L^{\infty}(B_{\hat
r_0}(x_0))}, \|f\|_{L^{\infty}(B_{\hat r_0}(x_0))},\ltwo, p_{\min},p_{\max}, N, a_0, a_1, L_1)$. Here we have used the bound in
\eqref{cota-v}.

Then, if we take $\ep\le\frac{p_{\min}}{N}$, we have by
\eqref{otra2} and by our election of $\rho$, that
$$
\begin{aligned}\pint_{B_{\rho}} |\nabla u|^{p_-}\, dx&\leq
\pint_{B_{\rho}} |\nabla u|^{p(x)}\,
dx+\frac{1}{|B_{\rho}|}\int_{B_{\rho}\cap\{|\nabla u|<1\}} |\nabla
u|^{p_-}\, dx \\&\leq \pint_{B_{\rho}} |\nabla u|^{p(x)}\, dx+ 1\\
&\leq 1+C \Big(\frac{r}{\rho}\Big)^N + C r^{-p_+}\\ & \leq
1+Cr^{-\ep N}+C r^{-p_+}
\\ & \leq  C r^{-p_+}=C \rho^{-\frac{p_+}{(1+\ep)}}.
\end{aligned}
$$

Now let  $r_0\le r_0(\ep,p_{\min},L)$ so that
$$\frac{p_+}{p_-}=\frac{p_+(B_{r_0})}{p_-(B_{r_0})}\leq 1+\frac{\ep}{2},$$
and small enough so that, in addition, $r_0^{\ep}\leq 1/2$. Then,
if $\rho\leq \rho_0=r_0^{1+\ep}$,
\begin{equation*}
\pint_{B_{\rho}} |\nabla u|^{p_-}\,
dx\leq  C \rho^{-\frac{(1+\frac{\ep}{2})}{(1+\ep)}p_-} = C
\rho^{-(1-\gamma)p_-},
\end{equation*}
where $\gamma=\frac{\frac{\ep}{2}}{(1+\ep)}=\gamma(N, p_{\min})$.
That is, if $\rho\leq \rho_0=r_0^{1+\ep}$
$$\Big(\pint_{B_{\rho}} |\nabla u|^{p_-}\,
dx\Big)^{1/p_-}\leq  C \rho^{\gamma-1}.$$

Thus \eqref{algosale} holds, with $C=C(L,
\|u\|_{L^{\infty}(B_{\hat r_0}(x_0))}, \|f\|_{L^{\infty}(B_{\hat
r_0}(x_0))},\ltwo, p_{\min},p_{\max}, N, a_0, a_1, L_1)$.

Applying Morrey's Theorem, see e.g. \cite{MZ}, Theorem 1.53, we
conclude that $u\in C^{\gamma}(B_{\rho_0}(x_0))$ and
$\|u\|_{C^{\gamma}(\overline{B_{\rho_0/2}(x_0)})}\leq {C}$ for
$C=C(\hat r_0, L, \|u\|_{L^{\infty}(B_{\hat r_0}(x_0))},
\|f\|_{L^{\infty}(B_{\hat r_0}(x_0))},\ltwo, p_{\min},p_{\max},
N, a_0, a_1, L_1)$.
\end{proof}

As a corollary we obtain

\begin{coro}\label{loc-holder}
Let $u$ be as in Theorem \ref{teo-holder}. Then $u\in
C^{\gamma}(\Omega)$ for some $0<\gamma<1$,  $\gamma=\gamma( N,
p_{\min})$. Moreover, if $\Omega'\subset\subset\Omega$, then
$\|u\|_{C^{\gamma}(\overline{\Omega'})}\leq C$ with $C$ depending
only on $N$, ${\rm dist}(\Omega',\partial\Omega)$, $p_{\min}$,
$p_{\max}$, $L$, $\ltwo$, $\|u\|_{L^{\infty}(\Omega)}$,
$\|f\|_{L^{\infty}(\Omega)}$, $a_0$, $a_1$ and $L_1$.
\end{coro}

Then, under the assumptions of the previous corollary we have that
$u$ is continuous in $\Omega$ and therefore, $\{u>0\}$ is open. We
can now prove the following property for nonnegative local
minimizers of \eqref{funct-J}

 \begin{lemm}\label{lemm-min-ig} Let $p, f$ and $\lambda$ be as in Theorem \ref{existence-minimizers-AC}.
 Let $u\in W^{1,p(\cdot)}(\Omega)\cap L^{\infty}(\Omega)$  be a nonnegative local minimizer of
\begin{equation*}
 J(v)=\int_\Omega\Big(\frac{|\nabla
 v|^{p(x)}}{p(x)}+\lambda(x)\chi_{\{v>0\}}+fv\Big)\,dx.
\end{equation*}
 Then
\begin{equation}\label{min-ig}
\Delta_{p(x)} u  = f \quad \mbox{ in }\{u>0\}.
\end{equation}
\end{lemm}
\begin{proof}
{ }From Lemma \ref{lemm-min-desig} we already know that
\eqref{min-desig} holds. In order to obtain the opposite
inequality in $\{u>0\}$, we let $0\le\xi\in C^{\infty}_0(\{u>0\})$
and consider $u-t\xi$, for $t<0$, with $|t|$ small.

 Using the minimality of $u$ we have
\begin{align*}0 &\geq \frac{1}{t} (J(u-t
\xi)-J(u)) = \frac{1}{t} \int_{\Omega}\Big(\frac{|\nabla u-t
\nabla \xi|^{p(x)}}{p(x)}-\frac{|\nabla u|^{p(x)}}{p(x)}\Big)\, dx
\,
-\int_{\Omega}f\xi\, dx\\
& \geq -\int_{\Omega}
 |\nabla u-t \nabla \xi|^{p(x)-2}(\nabla u-t \nabla
\xi)\cdot \nabla \xi\, dx -\int_{\Omega}f\xi\, dx
\end{align*}
and if we take $t\rightarrow 0$, we obtain
\begin{equation*}
0\geq -\int_{\Omega} |\nabla u|^{p(x)-2}\nabla u\cdot\nabla \xi\,
dx-\int_{\Omega}f\xi\, dx,
\end{equation*}
which gives  the desired inequality, so \eqref{min-ig} follows.
 \end{proof}

We will make use of the following version of Harnack's inequality

\begin{prop}\label{harnack} Let $x_0\in \R^N$ and $0<\delta\le 1$.  Let $1<p_{\min}\le p(x)\le p_{\max}<\infty$ in $B_{\delta}(x_0)$, with $\|\nabla
p\|_{L^{\infty}(B_{\delta}(x_0))}\leq L$ and $f\in {L^{\infty}(B_{\delta}(x_0))}$.  There exists a constant $C>0$ such that, if  $u\in W^{1,p(\cdot)}(B_{\delta}(x_0))\cap L^{\infty}(B_{\delta}(x_0))$  is a nonnegative  solution of
\begin{equation*}
\Delta_{p(x)} u  = f \quad \mbox{ in }B_{\delta}(x_0),
\end{equation*}
then,
\begin{equation}\label{des-harnack}\sup_{{B_{\frac{3}{4}\delta}}(x_0)} u\le C\big[\inf_{{B_{\frac{3}{4}\delta}}(x_0)}  u
+\delta\big].
\end{equation}
The constant $C$ depends only on $N$, $p_{\min}$, $p_{\max}$, $L$,  $\|f\|_{L^{\infty}(B_{\delta}(x_0))}$ and
${\|u\|_{L^{\infty}(B_{\delta}(x_0))}^{p_{+}^{\delta}-p_{-}^{\delta}}}$, where $p_{+}^{\delta}=\sup_{{B_{\delta}}(x_0)} p(x)$ and
$p_{-}^{\delta}=\inf_{{B_{\delta}}(x_0)} p(x)$.
\end{prop}
\begin{proof} We will first assume that $x_0=0$ and $\delta=1$. From Theorem 1.1 in \cite{Fan}, we know that $u\in C({B_1}(0))$.

Let $y_0\in B_{3/4}(0)$. Since  $\Delta_{p(x)} u  = f$ in $B_{1}(0)$, by Theorem 2.1 in \cite{Wo}, applied in  $B_{1/8}(y_0)$, we get
\begin{equation}\label{pre-harnack}\sup_{{B_{\frac{1}{32}}}(y_0)} u\le C\big[\inf_{{B_{\frac{1}{32}}}(y_0)}  u
+ 1\big],
\end{equation}
where $C$ is a positive constant that can be chosen so that $C>1$ and so that it depends only on $N$, $p_{\min}$, $p_{\max}$, $L$,
$\|f\|_{L^{\infty}(B_{1}(0))}$ and  ${\|u\|_{L^{\infty}(B_{1}(0))}^{p_{+}^{1}-p_{-}^{1}}}$, where $p_{+}^{1}=\sup_{{B_{1}}(0)} p(x)$ and
$p_{-}^{1}=\inf_{{B_{1}}(0)} p(x)$.

We now cover $\overline{B_{3/4}(0)}$ with $k$ balls  centered in ${B_{3/4}(0)}$ of radius $1/32$ ($k\ge 1$ a universal number).
Let $x,y\in \overline{B_{3/4}(0)}$, we choose balls of the covering and points, and we number them, in such a way that $x_0=x\in B_1$, $x_i\in B_i\cap B_{i+1}$ and
$x_j=y\in B_j$, for $1\le i \le j-1$ and $j\le k$.

It follows from \eqref{pre-harnack} that
$$ u(x_i)\le C\big[ u(x_{i+1}) + 1\big], \quad i=0,\dots,j-1,$$
which gives  $u(x)\le C^k\big[ u(y) + k\big]$. Therefore,
\begin{equation*}
\Delta_{p(x)} u  = f \quad \mbox{ in }B_{1}(0),
\end{equation*}
implies
\begin{equation}\label{pre-harnack-bis}\sup_{{B_{\frac{3}{4}}}(0)} u\le C\big[\inf_{{B_{\frac{3}{4}}(0)}}  u
+ 1\big],
\end{equation}
for a constant $C>0$ depending only on $N$, $p_{\min}$, $p_{\max}$, $L$,
$\|f\|_{L^{\infty}(B_{1}(0))}$ and  ${\|u\|_{L^{\infty}(B_{1}(0))}^{p_{+}^{1}-p_{-}^{1}}}$.

For general $x_0\in \R^N$ and $0<\delta\le 1$, we take $\bar u(x)=\frac{1}{\delta} u(x_0+\delta x)$. Then, as
\begin{equation*}
\Delta_{\bar p(x)} \bar u  = \bar f \quad \mbox{ in }B_{1}(0),
\end{equation*}
with $\bar p(x)= p(x_0+\delta x)$ and $\bar f(x)=\delta f(x_0+\delta x)$, there holds that $\bar u$ satisfies \eqref{pre-harnack-bis}. Finally, observing that $p_{\min}\le \bar p(x)\le p_{\max}$ in $B_{1}(0)$, $\|\nabla \bar p\|_{L^{\infty}(B_{1}(0))}\leq L$, $\|\bar f\|_{L^{\infty}(B_{1}(0))}\le \|f\|_{L^{\infty}(B_{\delta}(x_0))}$,
$$
{\|\bar u\|_{L^{\infty}(B_{1}(0))}^{\bar p_{+}^{1}-\bar p_{-}^{1}}}=
\Big({\frac{1}{\delta}}{\|u\|_{L^{\infty}(B_{\delta}(x_0))}}\Big)^{p_{+}^{\delta}-p_{-}^{\delta}},
$$
and
$$
\Big({\frac{1}{\delta}}\Big)^{p_{+}^{\delta}-p_{-}^{\delta}}\le \Big({\frac{1}{\delta}}\Big)^{2L{\delta}}\le C(L),
$$
we obtain the desired result.
\end{proof}

We will next prove the Lipschitz continuity of nonnegative local minimizers of \eqref{funct-J}. In the case in which $f\equiv 0$ and
$p(x)\ge 2$ this result was proven in \cite{FBMW}. In order to deal with the general case we will employ a different strategy than the one in
\cite{FBMW}.

\smallskip

Before getting the Lipschitz continuity we  prove  the following result

{\begin{theo}\label{pre-lip} Let $p, f, \lambda$ and $u$ be as in Lemma \ref{lemm-min-ig}.
 Let $\Omega'\subset\subset\Omega$. There exist  constants
 $C>0$, $r_0>0$ such that if  $x_0\in \Omega'\cap\partial\{u>0\}$ and $r\le r_0$ then
$$\sup_{B_{r}(x_0)} u\leq C r.$$
The constants depend only on $N, p_{\min}, p_{\max}, L,
||f||_{L^{\infty}(\Omega)}, \lone, \ltwo, ||u||_{
{L^{\infty}}(\Omega)}$ and ${\rm dist}(\Omega',\partial\Omega)$.
\end{theo}}
\begin{proof}
Let us suppose by contradiction that there exist a sequence of
nonnegative local minimizers $u_k$   corresponding to
functionals $J_k$ given by functions $p_k$, $f_k$ and $\lambda_k$,
with $u_k\in W^{1,p_k(\cdot)}(\Omega)\cap L^{\infty}(\Omega)$,  $p_{\min}\leq p_k(x)\leq p_{\max}$, $\|\nabla
p_k\|_{L^{\infty}}\leq L$, $||f_k||_{L^{\infty}(\Omega)}\le M_0$,
 $\lone\le\lambda_k(x)\le\ltwo$,  $||u_k||_{L^{\infty}(\Omega)}\le M$
and points $\bar x_k\in\Omega'\cap\partial\{u_k>0\}$,  such that
$$
\sup_{B_{r_k/4}(\bar x_k)} u_k\geq k r_k\quad \mbox{ and } \quad
r_k\le \frac{1}{k}.
$$

Without loss of generality we will assume that $\bar x_k=0$.

Let us define in $B_1$, for $k$ large,
$\bar u_k(x)=\frac1{r_k}u_k(r_k x)$, ${\bar p}_k(x)=p_k(r_k
x)$, ${\bar f}_k(x)=r_k f_k(r_k x)$ and ${\bar
\lambda}_k(x)={\lambda}_k(r_k x)$. Then $p_{\min}\leq {\bar
p}_k(x)\leq p_{\max}$, $\|\nabla {\bar
p}_k\|_{L^{\infty}(B_1)}\leq L r_k$,
$\lone\le{\bar\lambda}_k(x)\le\ltwo$ and $||{\bar
f}_k||_{L^{\infty}(B_1)}\le M_0 r_k$. Moreover, $\bar u_k$ is a
nonnegative minimizer in $\bar u_k + W_0^{1,\bar p_k(\cdot)}(B_1)$ of the functional
\begin{equation}\label{funcbarjk}
{\bar J}_k(v)=\int_{B_1}\Big(\frac{|\nabla v|^{{\bar
p}_k(x)}}{{\bar p}_k(x)}+{\bar\lambda}_k(x)\chi_{\{v>0\}}+{\bar
f}_k\,v\Big)\,dx
\end{equation}
with
$$\bar u_k(0)=0\qquad  \mbox{ and } \qquad \di\max_{\overline{B}_{1/4}}\bar u_k(x)>k.$$ Let
$d_k(x)=\di\mbox{dist}(x,\{\bar u_k=0\})$ and
$\mathcal{O}_k=\di\Big\{x\in B_1: d_k(x)\leq
\frac{1-|x|}{3}\Big\}$. Since $\bar u_k(0)=0$ then
$\overline{B}_{1/4}\subset \mathcal{O}_k$, therefore
$$m_k:=\sup_{\mathcal{O}_k}(1-|x|) \bar u_k(x)\geq
 \max_{\overline{B}_{1/4}}(1-|x|) \bar u_k(x)\geq \frac{3}{4} \max_{\overline{B}_{1/4}} \bar u_k(x)
 >\frac{3}{4} k.$$
 For each fix $k$, $\bar u_k$ is bounded, then $(1-|x|)
 \bar u_k(x)\rightarrow 0 \mbox{ when } |x|\rightarrow 1$ which means
 that there exists $x_k\in {\mathcal{O}_k}$ such that $(1-|x_k|)
 \bar u_k(x_k)= \sup_{\mathcal{O}_k}(1-|x|) \bar u_k(x)$, and then
\begin{equation}\label{mas34k}
\bar u_k(x_k)=\frac{m_k}{1-|x_k|}\geq m_k> \frac{3}{4} k
\end{equation}
as $x_k\in \mathcal{O}_k$, and $\delta_k:= d_k(x_k)\leq
\frac{1-|x_k|}{3}$. Let $y_k\in \partial\{\bar u_k>0\}\cap B_1$ such
that $|y_k-x_k|=\delta_k$. Then,
$$\begin{array}{ll}\di (1)\ B_{2\delta_k}(y_k)\subset B_1,\\\\ \di \mbox{ since if
} y\in B_{2\delta_k}(y_k) \Rightarrow |y|< 3\delta_k + |x_k|\leq
1,\\\\
\di (2)\ B_{\frac{\delta_k}{2}}(y_k)\subset \mathcal{O}_k, \\\di
\mbox{ since if } y\in B_{\frac{\delta_k}{2}}(y_k) \Rightarrow
|y|\leq \frac{3}{2}\delta_k + |x_k|\leq 1-\frac{3}{2} \delta_k
\Rightarrow d_k(y)\leq \frac{\delta_k}{2}\leq \frac{1-|y|}{3} \ \
\ \mbox{ and }\\\\\di (3) \mbox{ if } z\in
B_{\frac{\delta_k}{2}}(y_k) \Rightarrow 1-|z|\geq
1-|x_k|-|x_k-z|\geq 1-|x_k|-\frac{3}{2} \delta_k\geq
\frac{1-|x_k|}{2}.
\end{array}$$
By (2) we have
$$\max_{\mathcal{O}_k}(1-|x|) \bar u_k(x)\geq \max_{\overline{B_{\frac{\delta_k}{2}}}(y_k)}(1-|x|) \bar u_k(x)\geq
\max_{\overline{B_{\frac{\delta_k}{2}}}(y_k)}\frac{(1-|x_k|)}{2}
\bar u_k(x),$$ where in the last inequality we are using (3). Then,
\begin{equation}\label{uk1}2 \bar u_k(x_k)\geq
\max_{\overline{B_{\frac{\delta_k}{2}}}(y_k)}
\bar u_k(x).\end{equation} As $B_{\delta_k}(x_k)\subset \{\bar u_k>0\}$ then
$\Delta_{{\bar p}_k(x)}\bar u_k={\bar f}_k$ in $B_{\delta_k}(x_k)$, and by Harnack's
inequality (Proposition \ref{harnack}) we have
\begin{equation}\label{uk1-2}\max_{\overline{B_{\frac{3}{4}\delta_k}}(x_k)}\bar u_k(x)\le C\big[\min_{\overline{B_{\frac{3}{4}\delta_k}}(x_k)} \bar u_k(x)
+\delta_k\big],
\end{equation}
with $C$ a positive constant depending only on $N, p_{\min}, p_{\max}, L, M_0$ and $M$. We point out that, in order to get this uniform constant $C$ in
\eqref{uk1-2}, we have used, while applying Proposition \ref{harnack}, that
$$\gamma_k:=\sup_{B_{\delta_k}(x_k)}\bar p_k -\inf_{B_{\delta_k}(x_k)}\bar p_k\le 2L r_k\delta_k\le 2L r_k,$$
so that
$$||\bar u_k||_{L^{\infty}(B_{\delta_k}(x_k))}^{\gamma_k}\le ({M}/r_k)^{2L r_k}\le C_0(L,M).$$
Recalling \eqref{mas34k}, we get from \eqref{uk1-2}, for $k$ large,
\begin{equation}\label{uk2}\min_{\overline{B_{\frac{3}{4}\delta_k}}(x_k)} \bar u_k(x)\geq c
\bar u_k(x_k),\end{equation}
with $c$ a positive constant depending only on $N, p_{\min}, p_{\max}, L, M_0$ and $M$. As
$\overline{B_{\frac{3}{4}\delta_k}}(x_k)\cap
\overline{B_{\frac{\delta_k}{4}}}(y_k)\neq \emptyset$ we have by
\eqref{uk2}
\begin{equation}\label{uk3}
\max_{\overline{B_{\frac{\delta_k}{4}}}(y_k)} \bar u_k(x)\geq c
\bar u_k(x_k).\end{equation} Let
$w_k(x)=\di\frac{\bar u_k(y_k+\frac{\delta_k}{2} x)}{\bar u_k(x_k)}$. Then,
$w_k(0)=0$ and, by \eqref{uk1} and \eqref{uk3}, we have
\begin{align}\label{desiwk}\max_{\overline{B_1}}w_k\leq 2 &\qquad \max_{\overline{B_{1/2}}}
w_k\geq c>0.\end{align}
Now, recalling that $\bar u_k$ is a
nonnegative minimizer in $\bar u_k + W_0^{1,\bar p_k(\cdot)}(B_1)$ of the functional ${\bar J}_k$ in \eqref{funcbarjk} and that
$B_{\frac{\delta_k}{2}}(y_k)\subset B_1$,
we see that $w_k$ is a nonnegative minimizer of $\hat J_k$ in $w_k + W^{1,{\bar p}_k(y_k+\frac{\delta_k}{2} x)}_0(B_1)$, where
$$
{\hat J}_k(v)=\int_{B_1}\Big( c_k^{{\bar p}_k(y_k+\frac{\delta_k}{2} x)}
\frac{|\nabla v|^{{\bar p}_k(y_k+\frac{\delta_k}{2} x)}}{{\bar p}_k(y_k+\frac{\delta_k}{2} x)}+{\bar\lambda}_k(y_k+\frac{\delta_k}{2} x)\chi_{\{v>0\}}+{\bar f}_k(y_k+\frac{\delta_k}{2} x)\bar u_k(x_k)\,v\Big)\,dx,
$$
and $c_k=\frac{2 \bar u_k(x_k)}{\delta_k}$.

We now notice that $c_k\rightarrow\infty$. So we define ${{\tilde p}_k(x)}={\bar p}_k(y_k+\frac{\delta_k}{2} x)$ and divide the
functional $\hat J_k$ by $c_k^{{\tilde p}_k^{-}}$, with ${{\tilde p}_k^{-}}=\inf_{B_1}{\tilde p}_k$. Then, it follows that  $w_k$ is a nonnegative minimizer of $\tilde J_k$ in $w_k + W^{1,\tilde p_k(\cdot)}_0(B_1)$, where
$$
{\tilde J}_k(v)=\int_{B_1}\Big( \tilde a_k(x)\frac{|\nabla v|^{{\tilde
p}_k(x)}}{{\tilde p}_k(x)}+{\tilde\lambda}_k(x)\chi_{\{v>0\}}+{\tilde f}_k\,v\Big)\,dx,
$$
$\tilde a_k(x)=c_k^{{{\tilde p}_k(x)}-{{\tilde p}_k^{-}}}$ ,  ${\tilde\lambda}_k(x)={\bar\lambda}_k(y_k+\frac{\delta_k}{2} x)c_k^{-{{\tilde p}_k^{-}}}$
and ${\tilde f}_k(x)={\bar f}_k(y_k+\frac{\delta_k}{2} x)\bar u_k(x_k)c_k^{-{{\tilde p}_k^{-}}}$.

We claim that
\begin{equation}\label{fkacero}
\|{\tilde f}_k\|_{L^{\infty}}\leq \tilde M_0\quad \mbox{and}\quad {\tilde f}_k\rightarrow 0  \quad\mbox {uniformly in }B_1,
\end{equation}
\begin{equation}\label{lambdakacero}
{{\tilde\lambda}_k}\rightarrow 0  \quad\mbox {uniformly in }B_1,
\end{equation}
\begin{equation}\label{akauno}
{{\tilde a}_k}\rightarrow 1 \ \mbox{ uniformly,}  \quad 1\le {{\tilde a}_k}\le M_1\quad\mbox{and}\quad\|\nabla \tilde a_k\|_{L^{\infty}}\leq L_1\quad\mbox {in }B_1,
\end{equation}
\begin{equation}\label{pkapcero}
{\tilde p}_k\rightarrow p_0 \ \mbox{ uniformly}\quad\mbox{and}\quad p_{\min}\le p_0\le p_{\max} \quad\mbox {in }B_1,
\end{equation}
up to a subsequence, for some constants $\tilde M_0$, $M_1$, $L_1$ and $p_0$.

In fact, \eqref{fkacero} follows since
$|{\tilde f}_k(x)|=|r_k{f}_k(r_k(y_k+\frac{\delta_k}{2} x))\frac{u_k(r_k x_k)}{r_k}c_k^{-{{\tilde p}_k^{-}}}|\le M_0 M c_k^{-1}\to 0$. On the other hand, $0<{{\tilde\lambda}_k}(x)\le \ltwo c_k^{-1}\to 0$ gives \eqref{lambdakacero}.

In addition, in $B_1$  there holds, for $k$ large, that
$1\le {{\tilde a}_k}(x)\le e^{2 \|\nabla \tilde p_k\|_{L^{\infty}}\log c_k}$ and
$\|{\nabla{\tilde a}_k}\|_{L^{\infty}}\le \|\nabla \tilde p_k\|_{L^{\infty}}\log c_k\|{{\tilde a}_k}\|_{L^{\infty}}$. But
$\|\nabla \tilde p_k\|_{L^{\infty}}\log c_k\le L r_k\frac{\delta_k}{2}\log \big(\frac{2M}{r_k \delta_k}\big)\to 0$, which implies
\eqref{akauno}.

Finally, to see \eqref{pkapcero} we observe that $p_{\min}\leq p_k(x)\leq p_{\max}$ and $\|\nabla p_k\|_{L^{\infty}(\Omega)}\leq L$ and then,
for a subsequence, ${p}_k\rightarrow p$ uniformly on compacts of $\Omega$, so
${\tilde p}_k(x)={p}_k(r_k(y_k+\frac{\delta_k}{2} x))\rightarrow p_0=p(0)$ uniformly in $B_1$.

We now take $v_k$ the solution of
\begin{equation}\label{eqvk}
{\rm div}\big(\tilde a_k(x)|\nabla v_k|^{{\tilde p_k}(x)-2}\nabla v_k\big)={\tilde f}_k \quad\mbox{in }B_{3/4}, \qquad v_k-w_k\in
W_0^{1,\tilde p_k(\cdot)}(B_{3/4}).\end{equation}

{}From Lemma \ref{ppio-max-con-a}, Remark \ref{ppio-max-gral} and the bounds in \eqref{desiwk}, \eqref{fkacero} and \eqref{akauno}, it follows that if $k$ is large enough
\begin{equation}\label{cota-vk}
||v_k||_{L^{\infty}(B_{3/4})}\le \bar C \quad \mbox{ with } \quad\bar
C=\bar C(p_{\min}, \tilde M_0, L_1).
\end{equation}
Here we have used that $\|\nabla \tilde p_k\|_{L^{\infty}}\le L r_k\frac{\delta_k}{2}$ so $\|\nabla \tilde p_k\|_{L^{\infty}} 3/2< p_{\min}-1$
for $k$ large.

Then, applying Theorem 1.1 in \cite{Fan} we obtain  that, for $k$ large,
\begin{equation}\label{cota-vk-grad}
||v_k||_{C^{1,\alpha}(\overline{B_{1/2}})}\le \hat C \quad\mbox{ with }\quad \hat C=\hat C(p_{\min}, p_{\max}, \tilde M_0, L_1, L, M_1, N),
\end{equation}
for some $0<\alpha<1$. Therefore, there is a function
$v_0\in C^{1,\alpha}(\overline{B_{1/2}})$ such that, for a subsequence,
\begin{equation}\label{conv-vk}
v_k\rightarrow v_0    \quad\mbox{and}\quad \nabla v_k\rightarrow \nabla v_0\quad\mbox {uniformly in }\overline{B_{1/2}}.
\end{equation}
Moreover,  \eqref{fkacero}, \eqref{akauno} and \eqref{pkapcero} imply that
\begin{equation}\label{ec-v0}
\Delta_{p_0} v_0=0  \quad\mbox {in } B_{1/2}.
\end{equation}
Let us now show that
\begin{equation}\label{wk-vk-acero}
w_k-v_k\to 0   \quad\mbox {in } L^{p_{\min}}(B_{3/4}).
\end{equation}

{}From the minimality of $w_k$ we have
\begin{equation}\label{uso-min}
\int_{B_{3/4}} \tilde a_k(x)\frac{|\nabla w_k|^{\tilde p_k(x)}}{\tilde p_k(x)}- \tilde a_k(x)\frac{|\nabla v_k|^{\tilde p_k(x)}}{\tilde p_k(x)}+ \int_{B_{3/4}}{\tilde f}_k({w_k}-v_k)
\le C(N)\|\tilde {\lambda}_k\|_{L^{\infty}(B_{3/4})}.
\end{equation}
Then, we can argue as in the proof of Theorem \ref{teo-holder} and get estimate \eqref{standard-hold} for $u=w_k$, $v=v_k$,
$a(x)=\tilde a_k(x)$, $p(x)=\tilde p_k(x)$,  $f={\tilde f}_k$ and $r=3/4$, which together with \eqref{uso-min}, gives
\begin{align}\label{cota-previa-gradiente-1}
&\int_{A_2^k} |\nabla w_k-\nabla v_k|^{\tilde p_k(x)}  \, dx\leq
C \|\tilde {\lambda}_k\|_{L^{\infty}(B_{3/4})},\\\label{cota-previa-gradiente-2}
&\int_{A_1^k}|\nabla w_k-\nabla v_k|^2(|\nabla
w_k|+|\nabla v_k|)^{\tilde p_k(x)-2} \, dx\leq C \|\tilde {\lambda}_k\|_{L^{\infty}(B_{3/4})},
\end{align}
where $A_1^k=B_{3/4}\cap\{\tilde p_k(x)<2\}$, $A_2^k=B_{3/4}\cap\{\tilde p_k(x)\geq 2\}$ and $C=C(p_{\min},p_{\max},N)$.

Applying H\"older's inequality (Theorem \ref{holder}) with exponents $\frac{2}{\tilde p_k(x)}$ and $\frac{2}{2-\tilde p_k(x)}$, we get
\begin{equation}\label{aplic-holder1}
\int_{A_1^k} |\nabla w_k-\nabla v_k|^{\tilde p_k(x)}  \, dx\leq 2 \ \| F_k \|_{L^{2/{\tilde p_k(\cdot)}}({A_1^k})}
\|G_k\|_{L^{{2}/({2-\tilde p_k(\cdot)})}({A_1^k})},
\end{equation}
where
\begin{equation*}
\begin{aligned}
&F_k= |\nabla w_k-\nabla v_k|^{\tilde p_k}(|\nabla
w_k|+|\nabla v_k|)^{({\tilde p_k-2})\tilde p_k/2}\\
&G_k= (|\nabla w_k|+|\nabla v_k|)^{({2-\tilde p_k})\tilde p_k/2}.
\end{aligned}
\end{equation*}
Since
\begin{equation*}
\int_{A_1^k} |F_k|^{2/{\tilde p_k(x)}} \, dx= \int_{A_1^k}|\nabla w_k-\nabla v_k|^2(|\nabla
w_k|+|\nabla v_k|)^{\tilde p_k(x)-2} \, dx,
\end{equation*}
then,  from \eqref{cota-previa-gradiente-2},  \eqref{lambdakacero} and Proposition \ref{equi},  we get, for $k$ large,
\begin{equation}\label{cota-Fk}
\| F_k \|_{L^{2/{\tilde p_k(\cdot)}}({A_1^k})} \le
C \|\tilde {\lambda}_k\|_{L^{\infty}(B_{3/4})}^{p_{\min}/2},
\end{equation}
$C=C(p_{\min},p_{\max},N)$. On the other hand, \eqref{eqvk}  and the bounds in \eqref{fkacero}, \eqref{akauno} and \eqref{cota-vk}  give
\begin{equation*}
\begin{aligned}
\frac{1}{p_{\max}}
\int_{B_{3/4}} |\nabla v_k|^{\tilde p_k(x)} \leq&
\int_{B_{3/4}} \tilde a_k(x)\frac{|\nabla v_k|^{\tilde p_k(x)}}{\tilde p_k(x)} \\
\leq& \int_{B_{3/4}}\tilde a_k(x)\frac{|\nabla w_k|^{\tilde p_k(x)}}{\tilde p_k(x)}+ \int_{B_{3/4}}{\tilde f}_k({w_k}-v_k)\\
\leq& C\big(1+\int_{B_{3/4}}|\nabla w_k|^{\tilde p_k(x)}\big).
\end{aligned}
\end{equation*}
This implies
\begin{equation}\label{cota-Gk1}
\int_{A_1^k} |G_k|^{{2}/({2-\tilde p_k(x)})} \, dx\le C
\int_{B_{3/4}}(|\nabla w_k|^{\tilde p_k(x)}+|\nabla v_k|^{\tilde p_k(x)}) \, dx \le \tilde C\big(1+\int_{B_{3/4}}|\nabla w_k|^{\tilde p_k(x)}\big),
\end{equation}
for some $\tilde C=\tilde C(p_{\min},p_{\max},\tilde M_0, M_1, L_1)\ge 1$.
Now \eqref{cota-Gk1} and Proposition \ref{equi} give
\begin{equation}\label{cota-Gk2}
\| G_k \|_{L^{{2}/({2-\tilde p_k(\cdot)})}({A_1^k})} \le \tilde C\big(1+\int_{B_{3/4}}|\nabla w_k|^{\tilde p_k(x)}\big).
\end{equation}

Let us show that the right hand side in \eqref{cota-Gk2} can be bounded independently of $k$.

In fact, let $\tilde v_k$ be the solution of
\begin{equation}\label{eq-tilde-vk}
{\rm div}\big(\tilde a_k(x)|\nabla \tilde v_k|^{{\tilde p_k}(x)-2}\nabla \tilde v_k\big)={\tilde f}_k \quad\mbox{in }B_{7/8},
\qquad \tilde v_k-w_k\in
W_0^{1,\tilde p_k(\cdot)}(B_{7/8}).\end{equation}
Then, similar arguments to those leading to \eqref{cota-vk} and  \eqref{cota-vk-grad}, give, for $k$ large enough,
\begin{equation}\label{cota-tilde-vk}
||\tilde v_k||_{L^{\infty}(B_{7/8})}\le \bar C \quad \mbox{ with } \quad\bar
C=\bar C(p_{\min}, \tilde M_0, L_1),
\end{equation}
and
\begin{equation}\label{cota-tilde-vk-grad}
||\tilde v_k||_{C^{1,\alpha}(\overline{B_{3/4}})}\le \hat C \quad\mbox{ with }\quad \hat C=\hat C(p_{\min}, p_{\max}, \tilde M_0, L_1, L, M_1, N),
\end{equation}
for some $0<\alpha<1$.

Since $w_k$ is a nonnegative minimizer of $\tilde J_k$ in $B_1$, then we can argue as in the proof of Theorem \ref{teo-holder} and get estimate \eqref{otra} for $u=w_k$, $v=\tilde v_k$,
$a(x)=\tilde a_k(x)$, $p(x)=\tilde p_k(x)$, $\lambda(x)=\tilde\lambda_k(x)$, $f={\tilde f}_k$, $r=7/8$ and $\rho=3/4$. That is,
\begin{equation}\label{otra-wk}\int_{B_{3/4}} |\nabla w_k|^{\tilde p_k(x)}\, dx
 \leq {C} +C\int_{B_{3/4}} |\nabla \tilde v_k|^{\tilde p_k(x)}\, dx,\end{equation}
where $C=C(p_{\min},p_{\max},N, \ltwo)$. Therefore  \eqref{otra-wk} and \eqref{cota-tilde-vk-grad} give, for $k$ large, a uniform bound for the
right hand side in \eqref{cota-Gk2}. That is,
\begin{equation}\label{cota-Gk2-b}
\| G_k \|_{L^{{2}/({2-\tilde p_k(\cdot)})}({A_1^k})} \le \bar C,
\end{equation}
with $\bar C=\bar C(p_{\min}, p_{\max}, \tilde M_0, L_1, L, M_1,N, \ltwo)$.

Now, putting together \eqref{cota-previa-gradiente-1}, \eqref{aplic-holder1}, \eqref{cota-Fk}, \eqref{cota-Gk2-b} and \eqref{lambdakacero}, we obtain
\begin{equation}
\int_{B_{3/4}} |\nabla w_k-\nabla v_k|^{\tilde p_k(x)}  \to 0.
\end{equation}
Thus, using Poincare's inequality (Theorem  \ref{poinc} ) and Theorem \ref{imb}, we get \eqref{wk-vk-acero}.

 In order to conclude the proof, we now observe that,  by Corollary \ref{loc-holder}, there exists $0<\gamma<1$,  $\gamma=\gamma( N,
p_{\min})$, such that
\begin{equation*}\|w_k\|_{C^{\gamma}(\overline{B_{1/2}})}\leq C \ \quad \mbox{ with } \quad
C=C(p_{\min}, p_{\max}, \tilde M_0, L_1, L, M_1, N,\ltwo)\end{equation*}
(recall that $\|w_k\|_{L^\infty(B_1)}\le 2$).

Therefore, there is a function
$w_0\in C^{\gamma}(\overline{B_{1/2}})$ such that, for a subsequence,
\begin{equation}\label{conv-wk}
w_k\rightarrow w_0   \quad\mbox {uniformly in }\overline{B_{1/2}}.
\end{equation}

In addition, recalling \eqref{conv-vk}, \eqref{ec-v0} and \eqref{wk-vk-acero}, we get $v_0=w_0$ in $\overline{B_{1/2}}$ and
$\Delta_{p_0} w_0=0$ in $B_{1/2}$.

Finally, since there holds that  $w_k\ge 0$, $w_k(0)=0$ and \eqref{desiwk}, now \eqref{conv-wk} implies
\begin{equation*}
w_0\ge 0,  \quad w_0(0)=0, \quad \max_{\overline{B_{1/2}}} w_0\geq c>0,
\end{equation*}
 which contradicts the strong minimum principle and concludes the proof.
\end{proof}

We can now prove the Lipschitz continuity of nonnegative local minimizers

\begin{coro}\label{Lip}  Let $p, f, \lambda$ and $u$ be as in Lemma \ref{lemm-min-ig}.
 Then $u$ is locally Lipschitz
continuous in $\Omega$. Moreover, for any
$\Omega'\subset\subset
 \Omega$  the Lipschitz constant of $u$ in $\Omega'$ can be
estimated by a constant $C$ depending
only on $N$,  $p_{\min}$,
$p_{\max}$, $L$,  $\lone$, $\ltwo$, $\|u\|_{L^{\infty}(\Omega)}$,
$\|f\|_{L^{\infty}(\Omega)}$ and ${\rm dist}(\Omega',\partial\Omega)$.
\end{coro}

\begin{proof}
The result is a consequence of Corollary \ref{loc-holder}, Lemma \ref{lemm-min-ig} and  Theorem \ref{pre-lip} above, and Proposition 2.1 in  \cite{LW5}.
\end{proof}

Next we have

\begin{theo}\label{nodegdist} Let $p, f, \lambda$ and $u$ be as in Lemma \ref{lemm-min-ig}. Assume moreover that
  $\nabla u\in L^{\infty}(\Omega)$. There exist positive constants
$c_0$ and $\rho$  such that, for every $x\in\Omega'$,
\begin{equation*}
 u(x)\ge c_0 {\rm dist}(x,\{u\equiv 0\}), \quad \mbox{ if } \ {\rm dist}(x,\{u\equiv 0\})\le \rho.
  \end{equation*}
The constants depend only on $p_{\min}, p_{\max}, L,
||f||_{L^{\infty}(\Omega)}, \lone, \ltwo, ||\nabla u||_{
{L^{\infty}}(\Omega)}$ and ${\rm dist}(\Omega',\partial\Omega)$.
 \end{theo}
\begin{proof}  We will prove the statement for $x\in\Omega'$ such that $u(x)>0$ (otherwise there is nothing to prove).
Let us suppose by contradiction that there exist a sequence of
nonnegative local minimizers $u_k\in W^{1,p_k(\cdot)}(\Omega)\cap L^{\infty}(\Omega)$ corresponding to
functionals $J_k$ given by functions $p_k$, $f_k$ and $\lambda_k$,
with $p_{\min}\leq p_k(x)\leq p_{\max}$, $\|\nabla
p_k\|_{L^{\infty}}\leq L$, $||f_k||_{L^{\infty}(\Omega)}\le L_1$,
 $\lone\le\lambda_k(x)\le\ltwo$,  $||\nabla u_k||_{L^{\infty}(\Omega)}\le L_2$
and points $x_k\in\Omega'$, with $u_k(x_k)>0$,  such that
$$
d_k={\rm dist}(x_k,\{u_k\equiv 0\})\to 0 \quad \mbox{ and } \quad
\frac{u_k(x_k)}{d_k}\to 0.
$$
Let us define in $B_1$, for $d_k$ small,
$w_k(x)=\frac1{d_k}u_k(x_k+d_k x)$, ${\bar p}_k(x)=p_k(x_k+d_k
x)$, ${\bar f}_k(x)=d_k f_k(x_k+d_k x)$ and ${\bar
\lambda}_k(x)={\lambda}_k(x_k+d_k x)$. Then $p_{\min}\leq {\bar
p}_k(x)\leq p_{\max}$, $\|\nabla {\bar
p}_k\|_{L^{\infty}(B_1)}\leq L d_k$,
$\lone\le{\bar\lambda}_k(x)\le\ltwo$ and $||{\bar
f}_k||_{L^{\infty}(B_1)}\le L_1 d_k$. Moreover, $w_k$ is a
nonnegative local minimizer of the functional
$$
{\bar J}_k(v)=\int_{B_1}\Big(\frac{|\nabla v|^{{\bar
p}_k(x)}}{{\bar p}_k(x)}+{\bar\lambda}_k(x)\chi_{\{v>0\}}+{\bar
f}_k\,v\Big)\,dx.
$$
Since $w_k>0$ in $B_1$, we have $\Delta_{{\bar p}_k(x)}w_k={\bar
f}_k$ in $B_{1}$ (see \eqref{min-ig}). In addition,
$w_k(0)=\frac{u_k(x_k)}{d_k}\to 0$ and $||\nabla
w_k||_{L^{\infty}(B_1)}\le L_2$. Then, by interior H\"older
gradient estimates it follows that, for a subsequence, $w_k\to
w_0$ and $\nabla w_k\to \nabla w_0$ uniformly on compact subsets
of $B_1$. Moreover, for a subsequence, ${\bar f}_k\to 0$ and
${\bar p}_k\to p_0$ uniformly on compact subsets of $B_1$, with
$p_0$ constant. This implies that $\Delta_{p_0} w_0=0$  in $B_1$.

By Harnack's inequality there exists a constant $\overline c>0$,
depending on $N$ and $p_0$, such that
$$
 {\sup}_{B_{1/2}} w_0\le {\overline c}\, {\inf}_{B_{1/2}} w_0
$$
and therefore, given $\delta>0$, there exists $k_0$ such that for
$k\ge k_0$
$$
 {\sup}_{B_{1/2}} w_k\le {\overline c}\, {\inf}_{B_{1/2}} w_k + C_0\delta,
$$
for a constant $C_0$ depending on $N$ and $p_0$. In particular we
have, for $k$ large,
$$
 w_k(x)\le {\overline c}\, w_k(0)+ C_0 \delta\quad \mbox{in }B_{1/2}.
$$

Let ${\alpha}_k>0$ be such that $u_k(x_k)={\alpha}_k d_k$, this
is,  ${\alpha}_k=w_k(0)$. Let $\psi\in C^\infty(\overline B_1)$
such that $\psi\equiv0$ in $B_{1/4}$, $\psi\equiv1$ in $\overline
B_1\setminus B_{1/2}$, $0\le\psi\le 1$ and let
$$
z_k(x)=\begin{cases}
\mbox{min\,}\Big(w_k(x),(\overline c{\alpha}_k+C_0 \delta)\psi\Big)\quad&\mbox{in }B_{1/2},\\
w_k(x)\quad&\mbox{outside  }B_{1/2}.
\end{cases}
$$
Then, $z_k\in W^{1,{\bar p}_k(\cdot)}(B_{1})$ and $z_k$ coincides
with $w_k$ on $\partial B_1$ so that there holds that ${\bar
J}_k(z_k)\ge{\bar J}_k(w_k)$.

Let ${\mathcal D}_k=B_{1/2}\cap\{w_k>(\overline c{\alpha}_k+C_0
\delta)\psi\}$. Observe that $z_k\le w_k$, so that
$\chi_{\{z_k>0\}}\le \chi_{\{w_k>0\}}$. In addition, $w_k>0$ in
$B_{1/4}$, $z_k=0$ in $B_{1/4}$ and $B_{1/4}\subset{\mathcal
D}_k$. Therefore, if $C_0 \delta\le\frac12$ and $k$ is large
enough
 so that $\overline c{\alpha}_k\le\frac12$, we get
\begin{align*}
&\lone|B_{1/4}|\le\int_{{\mathcal D}_k}{\bar\lambda}_k(x)\big\{\chi_{\{w_k>0\}}-\chi_{\{z_k>0\}}\big\}\,dx\\
&\le\int_{{\mathcal D}_k}\frac{(\overline c{\alpha}_k+C_0
\delta)^{p_{\min}}}{p_{\min}}|\nabla\psi|^{{\bar p}_k} +L_1
d_k\int_{{\mathcal D}_k}\big[(\overline c{\alpha}_k+C_0
\delta)\psi+w_k\big]\,dx\le C (\overline c{\alpha}_k+C_0 \delta),
\end{align*}
with $C=C(\psi,p_{\min}, p_{\max}, L_1)$. So that
\begin{equation*}
\lone|B_{1/4}|\le  C(\overline c{\alpha}_k+C_0 \delta),
\end{equation*}
and, if  $C C_0\delta\le \frac{1}{2}\lone|B_{1/4}|$, it
follows that
$$
\frac{1}{2}\lone|B_{1/4}|\le {\bar C}{\alpha}_k ={\bar
C}\frac{u_k(x_k)}{d_k}\to 0,
$$
which is a contradiction.
\end{proof}

We also have

\begin{lemm}\label{itera1} Let $p$ and $f$ be as in Theorem \ref{existence-minimizers-AC}.
Let
$\Omega'\subset\subset\Omega$ and  $u\in C(\Omega)$, $u\ge 0$,
$\nabla u \in L^{\infty}(\Omega)$ with $\Delta_{p(x)} u  = f$ in
$\{u>0\}$ be such  that there exist positive constants $c_0$ and
$\rho$  such that, for every $x\in\Omega'$, there holds that
$u(x)\ge c_0 {\rm dist}(x,\{u\equiv 0\})$ if ${\rm
dist}(x,\{u\equiv 0\})\le \rho$. Then, there exist positive
constants $\delta_0$ and $\rho_0$
 such that for every $x\in\Omega'\cap\{u>0\}$ with $d(x)={\rm dist}(x,\{u\equiv 0\})\le \rho_0$,
we have
$$
\sup_{B_{d(x)}(x)} u \geq (1+\delta_0) u(x).
$$
The constants depend only on $p_{\min}, p_{\max}, L,
||f||_{L^{\infty}(\Omega)}, ||\nabla u||_{L^{\infty}(\Omega)},
c_0,\rho$ and ${\rm dist}(\Omega',\partial\Omega)$.
\end{lemm}
\begin{proof}
Suppose by contradiction that there exist  functions $u_k$, $p_k$,
$f_k$, with $1<p_{\min}\le p_k(x)\le p_{\max}<\infty$, $\|\nabla
p_k\|_{L^{\infty}}\leq L$, $||f_k||_{L^{\infty}(\Omega)}\le L_1$,
$u_k\in C(\Omega)$,  $u_k\ge 0$, $||\nabla
u_k||_{L^{\infty}(\Omega)}\leq L_2$, with $\Delta_{p_k(x)} u_k  =
f_k$ in $\{u_k>0\}$ and
 $u_k(x)\ge c_0 {\rm dist}(x,\{u_k\equiv 0\})$ if ${\rm dist}(x,\{u_k\equiv 0\})\le \rho$ and $x\in\Omega'$, and sequences
 $\delta_k\to 0$, $\rho_k\to 0$ and $x_k\in \Omega'\cap\{u_k>0\}$ with
$d_k={\rm dist}(x_k,\{u_k\equiv 0\})\le\rho_k$ such that
$$
\sup _{ B_{d_k}(x_k)} u_k\leq (1+\delta_k) u_k(x_k).$$

Take $w_k(x)=\displaystyle\frac{u_k(x_k+d_k x)}{u_k(x_k)}$. Then,
$w_k(0)=1$ and
$$
\max_{\overline{B_1}} w_k\leq (1+\delta_k), \quad w_k>0\quad
\mbox{ and }\quad{\rm
div}\Big(\Big(\frac{u_k(x_k)}{d_k}\Big)^{{\bar p}_k(x)-1}|\nabla
w_k|^{{\bar p}_k(x)-2}\nabla w_k\Big)={\bar f}_k\quad\mbox{in
}B_1,
$$
where ${\bar p}_k(x)=p_k(x_k+d_k x)$ and ${\bar f}_k(x)=d_k
f_k(x_k+d_k x)$. On the other hand,  we have
$$c_0\le\frac{u_k(x_k)}{d_k}\leq L_2, \qquad \|\nabla w_k\|_{L^{\infty}(B_1)}\le L_2\frac{d_k}{u_k(x_k)}\leq\frac{L_2}{c_0}.$$

Then, using the gradient estimates in  \cite{Fan}, we deduce that,
for a subsequence, $\frac{u_k(x_k)}{d_k}\to a\in[c_0,L_2]$, $w_k
\to \overline{w}$ and ${\bar p}_k\to p_0\in \R$ uniformly in
$\overline{B}_1$ and $\nabla w_k \to \nabla\overline{w}$
uniformly on compact subsets of ${B}_1$.

There holds that $\Delta_{p_0}\overline{w}=0$ in ${B}_1$,
$\overline{w}(0)=1$ and $\overline{w}\le 1$ in ${B}_1$. Therefore
$\overline{w}\equiv 1$ in $\overline B_1$.

Let $y_k\in \partial \{u_k>0\}$ with $|x_k-y_k|=d_k$. Then,
 if
$z_k=\frac{y_k-x_k}{d_k}$, we have
$$w_k(z_k)=\frac{u_k(y_k)}{u_k(x_k)}=0$$ and we may assume that
$z_k\to \bar{z} \in \partial B_1$. Thus,
$1=\overline{w}(\bar{z})=0$. This is a contradiction, and the
lemma is proved.
\end{proof}

As a consequence of the previous results, we obtain

\begin{theo}\label{nodegfinal}
Let $p, f, \lambda$ and $u$ be as in Theorem \ref{nodegdist}.
 Let $\Omega'\subset\subset\Omega$. There exist  constants
 $c>0$, $r_0>0$ such that if  $x_0\in \Omega'\cap\partial\{u>0\}$ and $r\le r_0$ then
$$\sup_{B_{r}(x_0)} u\geq c r.$$
The constants depend only on $p_{\min}, p_{\max}, L,
||f||_{L^{\infty}(\Omega)}, \lone, \ltwo, ||\nabla u||_{
{L^{\infty}}(\Omega)}$ and ${\rm dist}(\Omega',\partial\Omega)$.
\end{theo}
\begin{proof} We will follow the ideas of Theorem 1.9 in \cite{CS}.

\noindent{Step 1.} We will prove that there exist positive
constants
 $\bar c$, $\bar r$ and $\bar \rho$ such that if  $x_0\in \Omega'\cap\{u>0\}$, ${\rm dist}(x_0,\{u\equiv 0\})\le \bar\rho$ and $r\le \bar r$, then
$$\sup_{B_{r}(x_0)} u\geq {\bar c} r.$$

In fact, let $\rho_1={\rm dist}(\Omega',\partial\Omega)$ and
$\tilde\Omega=B_{{\rho_1}/2}(\Omega')$, so
$\Omega'\subset\subset\tilde\Omega\subset\subset\Omega$.

By Theorem \ref{nodegdist} and Lemma \ref{itera1} (applied to
points in $\tilde\Omega$), there exist positive constants $c_0$
and $\rho$  such that, for every $x\in\tilde\Omega$ with ${\rm
dist}(x,\{u\equiv 0\})\le \rho$,
\begin{equation*}
 u(x)\ge c_0 {\rm dist}(x,\{u\equiv 0\}),
\end{equation*}
and positive constants $\delta_0$ and $\rho_0$
 such that for every $x\in\tilde\Omega\cap\{u>0\}$ with $d(x)={\rm dist}(x,\{u\equiv 0\})\le \rho_0$,
we have
$$
\sup_{B_{d(x)}(x)} u \geq (1+\delta_0) u(x).
$$
The constants depend only on $p_{\min}, p_{\max}, L,
||f||_{L^{\infty}(\Omega)}$, $||\nabla u||_{
{L^{\infty}}(\Omega)}$,   ${\rm
dist}(\tilde\Omega,\partial\Omega)=\frac12{\rm
dist}(\Omega',\partial\Omega)$, $\lone$ and $\ltwo$.

Let $\bar r=\min\{\frac12{\rm
dist}(\Omega',\partial\tilde\Omega),\rho, \rho_0\}$,
$\bar\rho=\rho$ and $r\le \bar r$. Let  $x_0\in
\Omega'\cap\{u>0\}$ such that  $d_0={\rm dist}(x_0,\{u\equiv
0\})\le \bar\rho$, then
\begin{equation*}
 u(x_0)\ge c_0 d_0.
\end{equation*}
There are two possibilities:

\item[i)] $d_0\ge \frac{r}{8}$.

In this case  $u(x_0)\ge c_0 \frac{r}{8}$ and the result follows.
\item[ii)] $d_0< \frac{r}{8}$.

In this case, proceeding as in \cite{CS}, we construct a polygonal
that never leaves $B_r(x_0)$, starting at $x_0$  and finishing at
$\tilde x\in B_r(x_0)$, such that $u(\tilde x)\ge \tilde c r$,
with  an explicit $\tilde c>0$ depending on the constants
mentioned above. We refer to \cite{CS} for the details. In the
present situation, the mean value argument employed in \cite{CS}
is replaced by the argument in Lemma \ref{itera1}.

\noindent{Step 2.} Now let $\bar r$ and $\bar \rho$ as above,
$r\le \bar r$ and $x_0\in \Omega'\cap\partial\{u>0\}$. We take
$x_1\in B_{\frac{r}{2}}(x_0)\cap\{u>0\}\cap\Omega'$. Then, ${\rm
dist}(x_1,\{u\equiv 0\})\le |x_1-x_0|\le\bar\rho$ and thus, from
the result in Step 1,
$$
\sup_{B_{r}(x_0)} u\ge\sup_{B_{\frac{r}{2}}(x_1)} u\geq \bar c
\frac{r}{2}.
$$
This completes the proof.
\end{proof}

The following result in the section is

\begin{theo}\label{densprop} Let $p, f, \lambda$ and $u$ be as in Theorem \ref{nodegdist}.
Let $\Omega'\subset\subset\Omega$. There exist constants
$\tilde{c}\in (0,1)$ and ${\tilde r}_0>0$  such that, if
$x_0\in\Omega'\cap\partial\{u>0\}$ with $B_r(x_0)\subset \Omega'$
and $r\le {\tilde r}_0$, there holds
\begin{equation*}
 \frac{|B_r(x_0)\cap\{u>0\}|}{|B_r(x_0)|} \leq 1-\tilde{c}.
  \end{equation*}
  The constants depend only on $p_{\min}, p_{\max}, L, ||f||_{L^{\infty}(\Omega)}, \lone, \ltwo, ||\nabla u||_{L^{\infty}(\Omega)}$
  and ${\rm dist}(\Omega',\partial\Omega)$.
 \end{theo}
\begin{proof} Let us suppose by contradiction
that there exist a sequence of nonnegative local minimizers
$u_k\in  W^{1,p_k(\cdot)}(\Omega)\cap L^{\infty}(\Omega)$   corresponding to functionals $J_k$ given by
functions $p_k$, $f_k$ and $\lambda_k$, with $p_{\min}\leq
p_k(x)\leq p_{\max}$, $\|\nabla p_k\|_{L^{\infty}}\leq L$,
$||f_k||_{L^{\infty}(\Omega)}\le L_1$,
 $\lone\le\lambda_k(x)\le\ltwo$,  $||\nabla u_k||_{L^{\infty}(\Omega)}\le L_2$
 and balls $B_{r_k}(x_k)\subset\Omega'$ with $x_k\in\partial\{u_k>0\}$ and $r_k\to 0$, such
that
$$
\frac{|B_{r_k}(x_k)\cap\{u_k=0\}|}{|B_{r_k}(x_k)|}\rightarrow 0
$$
and
$$\sup_{B_{r_k\sigma}(x_k)} u_k \geq c r_k\sigma, \quad\mbox{ for } 0<\sigma <1, $$
where $c$ is the positive constant given by Theorem
\ref{nodegfinal}.

Let ${\bar u}_k(x)=\frac{u_k(x_k+r_k x)}{r_k}$, ${\bar
p}_k(x)=p_k(x_k+r_k x)$ and ${\bar f}_k(x)=r_k f_k(x_k+r_k x)$.
Then $p_{\min}\leq {\bar p}_k(x)\leq p_{\max}$, $\|\nabla {\bar
p}_k\|_{L^{\infty}(B_1)}\leq L r_k$, $||{\bar
f}_k||_{L^{\infty}(B_1)}\le L_1 r_k$,
 $0\in\partial\{{\bar u}_k>0\}$,
$$
|B_1\cap\{{\bar u}_k=0\}|=\ep_k\rightarrow 0,
$$
\begin{equation}\label{nodeg-uk}
\sup_{B_{\sigma}} {\bar u}_k \geq c \sigma, \quad\mbox{ for }
0<\sigma <1,
\end{equation}
and
\begin{equation*}
\Delta_{{\bar p}_k(x)}{\bar u}_k\ge{\bar f}_k\ \quad \mbox{ in }
B_{1/2}.
\end{equation*}

Let us take $v_k\in W^{1,{\bar p}_k(\cdot)}(B_{1/2})$, such that
\begin{equation}\label{ecvk}
\Delta_{{\bar p}_k(x)} v_k={\bar f}_k\ \quad \mbox{ in }
B_{1/2},\qquad v_k-{\bar u}_k\in W^{1,{\bar
p}_k(\cdot)}_0(B_{1/2}).
\end{equation}
Observe that there holds that $||{\bar
u}_k||_{L^{\infty}(B_{1/2})}\le L_2/2$ implying that
\begin{equation}\label{cota-ecvk}
||v_k||_{L^{\infty}(B_{1/2})}\le \bar C\ \quad \mbox{ with }
\quad\bar C=\bar C(L, p_{\min},  L_1, L_2),
\end{equation}
(this estimate follows from Lemma \ref{ppio-max-con-a} and Remark \ref{ppio-max-gral}, if $k$ is large enough).

Since $v_k\ge {\bar u}_k$ then $0\le\chi_{\{v_k>0\}}-\chi_{\{{\bar
u}_k>0\}}\le \chi_{\{{\bar u}_k=0\}}$ and therefore, using that
${\bar u}_k$ are nonnegative local minimizers, we get
\begin{equation}\label{baruk-min}
\int_{B_{1/2}}\Big(\frac{|\nabla {\bar u}_k|^{{\bar
p}_k(x)}}{{\bar p}_k(x)}-\frac{|\nabla v_k|^{{\bar p}_k(x)}}{{\bar
p}_k(x)}\Big) \le \ltwo |B_1\cap\{{\bar u}_k=0\}| +L_1r_k
\int_{B_{1/2}}|{\bar u}_k - v_k|.
\end{equation}

Applying \eqref{cota-ecvk}, we now obtain
\begin{equation}\label{cotadif}
\int_{B_{1/2}}\Big(\frac{|\nabla {\bar u}_k|^{{\bar
p}_k(x)}}{{\bar p}_k(x)}-\frac{|\nabla v_k|^{{\bar p}_k(x)}}{{\bar
p}_k(x)}\Big) \le C(\ep_k +L_1r_k).
\end{equation}

We claim that
\begin{equation}\label{wk}
\int_{B_{1/2}}|\nabla {\bar u}_k-\nabla v_k|^{{\bar
p}_k(x)}\,dx\to 0.
\end{equation}

In fact, let $u^s(x)=s {\bar u}_k(x)+(1-s) v_k(x)$. By using
\eqref{ecvk} and the inequalities in \eqref{desigualdades}, we get
\begin{equation*}\label{standard}\begin{aligned}
&\int_{B_{1/2}} \frac{|\nabla {\bar u}_k|^{{\bar p}_k(x)}}{{\bar p}_k(x)}- \frac{|\nabla v_k|^{{\bar p}_k(x)}}{{{\bar p}_k(x)}}+ \int_{B_{1/2}}{\bar f}_k({\bar u}_k-v_k)=\\
&\quad\int_0^1 \frac{ds}{s}\int_{B_{1/2}}\Big(|\nabla u^s|^{{{\bar
p}_k(x)}-2}\nabla u^s-
|\nabla v_k|^{{{\bar p}_k(x)}-2}\nabla v_k\Big) \cdot \nabla(u^s-v_k)\ge\\
&\quad\quad C\Big(\int_{B_{1/2}\cap\{{{\bar p}_k}\geq 2\} }
|\nabla {\bar u}_k-\nabla v_k|^{{\bar p}_k(x)}
+\int_{B_{1/2}\cap\{{{\bar p}_k}<2\} } |\nabla {\bar u}_k-\nabla
v_k|^2\Big(|\nabla {\bar u}_k|+|\nabla v_k|\Big)^{{{\bar
p}_k(x)}-2}\Big).
\end{aligned}
\end{equation*}
Now \eqref{cotadif} implies
$$
\begin{aligned}
& \int_{\{{\bar p}_k\geq 2\}\cap B_{1/2}} |\nabla {\bar
u}_k-\nabla v_k|^{{\bar p}_k(x)}\, dx\leq {\tilde C}(\ep_k
+L_1r_k)\quad \mbox{ and }
\\
& \int_{\{{\bar p}_k< 2\}\cap B_{1/2}}\big(|\nabla {\bar
u}_k|+|\nabla v_k|\big)^{{\bar p}_k(x)-2} |\nabla {\bar
u}_k-\nabla v_k|^2\, dx\leq {\tilde C}(\ep_k +L_1r_k).
\end{aligned}
$$
{}From these inequalities we obtain, reasoning as in the proof of
Theorem 5.1 in \cite{FBMW},
\begin{equation*}
\int_{B_{1/2}}|\nabla {\bar u}_k-\nabla v_k|^{{\bar
p}_k(x)}\,dx\le C\max\{\ep_k +L_1r_k,(\ep_k
+L_1r_k)^{{p_{\min}}/2}\}
\end{equation*}
and thus, \eqref{wk} follows.

 On the other hand, by interior H\"older gradient estimates, there holds that, for a subsequence,
$v_k\to v_0$ and $\nabla v_k\to\nabla v_0$ uniformly on compact
subsets of $B_{1/2}$. Since $\|\nabla {\bar
p}_k\|_{L^{\infty}(B_1)}\leq L r_k$, there exists a constant $p_0$
such that (for a subsequence) ${\bar p}_k\to p_0$ uniformly in
$B_{1/2}$.

Finally, since $ \|\nabla {\bar u}_k\|_{L^{\infty}(B_{1/2})} \leq
L_2$ we have, for a subsequence, ${\bar u}_k\to u_0$ uniformly in
$B_{1/2}$.

Let $w_k={\bar u}_k-v_k$. Then, $w_k\to u_0-v_0$ uniformly on
compact subsets of $B_{1/2}$. By \eqref{wk} we have that $\|\nabla
w_k\|_{L^{{\bar p}_k(\cdot)}(B_{1/2})}\to 0$. Since $w_k\in
W_0^{1,{\bar p}_k(\cdot)}(B_{1/2})$, by Poincare's inequality
(Theorem \ref{poinc}) we get that $\|w_k\|_{L^{{\bar
p}_k(\cdot)}(B_{1/2})}\to ||u_0-v_0||_{L^{p_0}(B_{1/2})}=0$.
 Thus, $u_0=v_0$.

Now, using  that $v_k\to u_0$ locally in $C^1(B_{1/2})$ and ${\bar
f}_k\to 0$ uniformly in $B_{1/2}$, we deduce that $\Delta_{p_0}
u_0=\Delta_{p_0} v_0=0$ in $B_{1/2}$.

As ${\bar u}_k\to u_0$ uniformly in $B_{1/2}$ we get, by
\eqref{nodeg-uk}, that $\sup_{B_{1/4}} u_0\geq \frac{c}{4}$.
 But $u_0(0)=\lim {\bar u}_k(0)=0$ and ${u}_0\ge 0$. By the  strong maximum principle  we arrive at a contradiction and the result follows.
\end{proof}

\medskip

We devote the last part of the section to discuss the fulfillment of properties (3) and (4) in the definition of weak solution for
nonnegative local minimizers of \eqref{funct-J}.

We need

\begin{defi}
\label{mild-minim} Let $p, f$ and $\lambda$ be as in Definition \ref{def-loc-min} and let $u\in W^{1,p(\cdot)+\delta_0}(\Omega)$, for some $\delta_0>0$. For an open set $D\subset \Omega$ let
\begin{equation*}
 J_D^{p,\lambda,f}(v)=J_D(v)=\int_D\Big(\frac{|\nabla
 v|^{p(x)}}{p(x)}+\lambda(x)\chi_{\{v>0\}}+fv\Big)\ dx.
\end{equation*}
We say that
$u$ is a mild minimizer of $J$
in $\Omega$ if for every $B_r(x_0)\subset\subset\Omega$ and $v\in W^{1,p(\cdot)+\delta}(B_r(x_0))$ with
$v-u\in W_0^{1,p(\cdot)+\delta}(B_r(x_0))$, for some $0<\delta<\delta_0$,
\begin{equation*}
 J_{B_r(x_0)}(u)\le J_{B_r(x_0)}(v).
\end{equation*}
\end{defi}

We have the following results for mild minimizers

\begin{prop}
\label{blow-up-mild-minim} Let $p, f$ and $\lambda$ be as in Theorem \ref{existence-minimizers-AC}.
Assume moreover that $\lambda\in C(\Omega)$.  Let $u$ be a nonnegative Lipschitz mild minimizer of $J$
in $\Omega$. Let $x_k\in\Omega\cap\partial\{u>0\}$, $x_k\to x_0\in\Omega$, $\rho_k\to 0$ and ${u}_k(x)=\frac{u(x_k+\rho_k x)}{\rho_k}$.
 Assume that $u_k\to u_0$ uniformly on compact sets of $\R^N$. Then $u_0$ is a nonnegative Lipschitz mild minimizer of $J$ in $\R^N$,
with $p(x)\equiv p(x_0)$, $\lambda(x)\equiv\lambda(x_0)$ and $f\equiv 0$.
\end{prop}
\begin{proof} Let
$B_r=B_r(\bar x_0)$ be any ball in $\R^N$ and  assume for simplicity that $\bar x_0=0$. Denote
${p}_k(x)=p(x_k+\rho_k x)$, $p_0=p(x_0)$, ${\lambda}_k(x)=\lambda(x_k+\rho_k x)$, $\lambda_0=\lambda(x_0)$,
${f}_k(x)=\rho_k f(x_k+\rho_k x)$ and
\begin{align*}
J_{r,k}(v)=\int_{B_r} \Big(\frac{|\nabla v|^{p_k(x)}}{p_k(x)}&+\lambda_k(x)\chi_{\{v>0\}}+f_k v\Big)\, dx,\\
J_{r,0}(v)=\int_{B_r}\Big(\frac{|\nabla  v|^{p_0}}{p_0}&+\lambda_0\chi_{\{v>0\}}\Big)\,dx.
\end{align*}
Let $v\in W^{1,p_0+\delta}(B_r)$ with $v-u_0\in W_0^{1,p_0+\delta}(B_r)$ for some $\delta>0$.
We want to show that
\begin{equation}\label{mild-minimiza}
J_{r,0}(u_0)\le J_{r,0}(v).
\end{equation}
For $h>0$ small, we define
$$
v_{h,k}=
\begin{cases}
v & \mbox{in }B_r,\\
u_0+\frac{|x|-r}{h}(u_k-u_0) & \mbox{in }B_{r+h} \setminus B_r.
\end{cases}
$$

Then, since $p_k\le p_0+\delta/2$ in ${B_{r+h}}$ for $k$ large, it follows that $v_{h,k}\in W^{1,p_k(\cdot)+\delta/2}(B_{r+h})$, $v_{h,k}-u_k\in W_0^{1,p_k(\cdot)+\delta/2}(B_{r+h})$, for $k$ large, and there
holds
\begin{align*}
&J_{r+h,k}(v_{h,k})= \int_{B_{r+h}} \Big(\frac{|\nabla v_{h,k}|^{p_k(x)}}{p_k(x)}+\lambda_k(x)\chi_{\{v_{h,k}>0\}}+f_k v_{h,k}\Big) =
\\
&\quad J_{r,0}(v)+\int_{B_{r+h}\setminus B_r} \Big(\frac{|\nabla v_{h,k}|^{p_k(x)}}{p_k(x)}+\lambda_k(x)\chi_{\{v_{h,k}>0\}}+f_k v_{h,k}\Big)+\\
&\quad\int_{B_r}\Big( \frac{|\nabla v|^{p_k(x)}}{p_k(x)}-\frac{|\nabla  v|^{p_0}}{p_0}+(\lambda_k(x)-\lambda_0)\chi_{\{v>0\}}+f_k v\Big)
\le \ J_{r,0}(v) + C_0hr^{N-1} + \\
&\quad\quad  C_1\int_{B_{r+h}\setminus B_r} \frac{|u_k-u_0|^{p_k(x)}}{h^{p_k(x)}}+
\int_{B_r}\Big( \frac{|\nabla v|^{p_k(x)}}{p_k(x)}-\frac{|\nabla  v|^{p_0}}{p_0}+(\lambda_k(x)-\lambda_0)\chi_{\{v>0\}}+f_k v\Big).
\end{align*}

Therefore,
\begin{equation}\label{mild-limsupvhk}
\limsup_{k\to\infty}J_{r+h,k}(v_{h,k})\le J_{r,0}(v)+ C_0hr^{N-1}.
\end{equation}

On the other hand,
$$
\lambda_0\chi_{\{u_0>0\}}\le\liminf_{k\to\infty} \lambda_k(x)\chi_{\{u_{k}>0\}},
$$
which implies
\begin{equation}\label{mild-primterm}
\int_{B_r} \lambda_0\chi_{\{u_0>0\}}\,dx \le
\liminf_{k\to\infty} \int_{B_r}\lambda_k(x)\chi_{\{u_{k}>0\}}\, dx.
\end{equation}

In addition, since
$\nabla u_k \rightharpoonup \nabla u_0$ weakly  in
$L^{p_0}({B_r})$,
arguing in a similar way as in Theorem \ref{existence-minimizers-AC}, we get
\begin{equation}\label{mild-segterm}
 \int_{B_r} \frac{|\nabla u_0|^{p_0}}{p_0}\, dx \le \liminf_{k\to\infty}
\int_{B_r}\frac{|\nabla u_k|^{p_0}}{p_0}\, dx=\liminf_{k\to\infty}
\int_{B_r}\frac{|\nabla u_k|^{p_k(x)}}{p_k(x)}\, dx.
\end{equation}

Now, using \eqref{mild-primterm} and \eqref{mild-segterm}, and the fact that
$u_k$ are  nonnegative Lipschitz mild minimizers of $J$ with $p(x)=p_k(x)$, $\lambda(x)=\lambda_k(x)$ and $f(x)=f_k(x)$  we obtain
$$
J_{r,0}(u_0)\le \liminf_{k\to\infty} J_{r,k}(u_k)\le \liminf_{k\to\infty} J_{r+h,k}(u_k)+ C_2hr^{N-1}
\le\liminf_{k\to\infty} J_{r+h,k}(v_{h,k})+ C_2hr^{N-1},
$$
which in combination with \eqref{mild-limsupvhk} gives
$$
J_{r,0}(u_0)\le J_{r,0}(v)+ C_3hr^{N-1}.
$$
Therefore, letting $h\to 0$ we obtain \eqref{mild-minimiza}.
\end{proof}

We will need

\begin{prop}
\label{identific-alpha} Let $1<p_0$ and $\lambda_0, \alpha$ be positive  constants. Let $u$  be a  Lipschitz mild minimizer of $J$
in $\R^N$,
with $p(x)\equiv p_0$, $\lambda(x)\equiv\lambda_0$ and $f\equiv 0$. Assume that $u=\alpha x_1^+$ in $B_{r_0}$, for some $r_0>0$.  Then, $\alpha=\Big(\frac{p_0}{p_0-1}\,\lambda_0\Big)^{1/p_0}$.
\end{prop}
\begin{proof} Let $\ep>0$ small,  let $\tau_{\ep}(x)=x+\ep \phi\displaystyle({|x|}) e_1$ with  $\phi\in C_0^{\infty}(-r_0,r_0)$,
and let $u_{\ep}(x)=u({\tau_{\ep}}^{-1}(x))$.

Then, $u_{\ep}\in W^{1,p_0+\delta}(B_{r_0})$ with $u_{\ep}-u\in W_0^{1,p_0+\delta}(B_{r_0})$, for some $\delta>0$, which implies that
$$
0\leq J_{{r_0},0}(u_{\ep})-J_{{r_0},0}(u),
$$
for
\begin{align*}
J_{r_0,0}(v)=\int_{B_{r_0}}\Big(\frac{|\nabla  v|^{p_0}}{p_0}&+\lambda_0\chi_{\{v>0\}}\Big)\,dx.
\end{align*}

We now proceed as in Lemma 7.3 in \cite{MW1}. In fact, there it is proved an analogous result with $J_{r_0,0}$ replaced by
$$
\mathcal{J}(v)=\int_{B_{r_0}} \Big(G(|\nabla v|) + \lambda \chi_{\{v>0\}}\Big)\,dx,
$$
for a general $G$ and a positive constant $\lambda$, and it is shown that
\begin{equation}\label{general-MW1}
G'(\alpha)\alpha-G(\alpha)=\lambda.
\end{equation}
Since in our case we have $\mathcal{J}$ with $G(t)=\frac{t^{p_0}}{p_0}$ and $\lambda=\lambda_0$, \cite{MW1} applies and thus \eqref{general-MW1} yields
\begin{equation*}
{\alpha}^{p_0}-\frac{\alpha^{p_0}}{p_0}=\lambda_0,
\end{equation*}
which gives the desired result.
\end{proof}

Next we prove

\begin{theo}
\label{limsup-grad} Let $p, f, \lambda$ and $u$ be as in Lemma \ref{lemm-min-ig}. Assume moreover that $\lambda\in C(\Omega)$.
Let $x_0\in\Omega\cap\partial\{u>0\}$. Then,
$$
\limsup_{\stackrel{x\to x_0}{u(x)>0}}\,|\nabla u(x)|=
\lambda^*(x_0),
$$
where $\lambda^*(x)=\Big(\frac{p(x)}{p(x)-1}\,\lambda(x)\Big)^{1/p(x)}$.
\end{theo}
\begin{proof}
Let
$$
\alpha:=\limsup_{\stackrel{x\to x_0}{u(x)>0}} |\nabla u(x)|.
$$
Since $u\in Lip_{\rm{loc}}(\Omega)$, $0\le\alpha<\infty$.
By the definition of $\alpha$ there exists a sequence
$z_k\rightarrow x_0$ such that
$$
u(z_k)>0,\quad \quad |\nabla u(z_k)|\rightarrow \alpha.
$$
Let $y_k$ be the nearest point from $z_k$ to $\Omega \cap
\partial\{u>0\}$ and let $d_k = |z_k-y_k|$.

Consider the blow up sequence $u_{d_k}$ with respect to
$B_{d_k}(y_k)$.  That is, $u_{d_k}(x)=\frac 1{d_k}u(y_k+d_k x)$.
Since $u$ is locally Lipschitz, and $u_{d_k}(0)=0$ for every $k$, there
exists $u_0$, with $u_0(0)=0$, such  that (for a subsequence)
$u_{d_k}\to u_0$ uniformly on compact sets of $\R^N$.
Moreover, using Lemma \ref{lemm-min-ig} and interior H\"older estimates we deduce that
$\nabla u_{d_k}\to \nabla u_0$ uniformly on compact subsets of $\{u_0>0\}$.

We claim that $|\nabla u_0|\leq \alpha$ in $\R^N$. In fact,
let $R>1$ and $\delta>0$. Then, there exists $\tau_0>0$ such that
$|\nabla u(x)|\leq \alpha+\delta$ for any $x\in B_{\tau_0
R}(x_0)$. For $|z_k-x_0|<\tau_0 R/2$ and $d_k<\tau_0/2$ we have
$B_{d_k R}(z_k)\subset B_{\tau_0 R}(x_0)$ and therefore, $|\nabla
u_{d_k}(x)|\leq \alpha +\delta$ in $B_{R-1}$ for $k$ large. Passing to
the limit, we obtain $|\nabla u_0|\leq \alpha+\delta$ in $B_{R-1}$,
and since $\delta$ and $R$ were arbitrary, the claim holds.

Now, if $\alpha=0$, since $u_0(0)=0$, it follows that $u_0\equiv 0$. This contradicts Theorem \ref{nodegfinal} and then, $\alpha>0$.

Next, define for $\gamma>0$, $(u_0)_{\gamma}(x)=\frac{1}{\gamma}
u_0(\gamma x)$. There exist a sequence $\gamma_n\to 0$ and
$u_{00}\in Lip(\R^N)$ such that $(u_0)_{\gamma_n}\to u_{00}$
uniformly on compact sets of $\R^N$.

Using Lemma \ref{lemm-min-ig} and Theorem \ref{densprop} and proceeding as in the proof of Theorem 5.1 in \cite{LW4} we
obtain that $u_{00}(x)=\alpha x_1^+$.

Now, since $u$ is a nonnegative local  minimizer of functional $J$ in $\Omega$, then $u$ is locally Lipschitz and  it is a nonnegative  mild minimizer of $J$ in $\Omega$. Thus, applying Proposition \ref{blow-up-mild-minim} to $u$ and to the blow up sequence $u_{d_k}$, we get that
$u_0$ is a nonnegative Lipschitz mild minimizer of $J$ in $\R^N$,
with $p(x)\equiv p(x_0)$, $\lambda(x)\equiv\lambda(x_0)$ and $f\equiv 0$.

Then, applying again Proposition \ref{blow-up-mild-minim}, now to $u_0$ and to the blow up sequence $(u_0)_{\gamma_n}$,
we also get that
$u_{00}(x)=\alpha x_1^+$ is a nonnegative Lipschitz mild minimizer of $J$ in $\R^N$,
with $p(x)\equiv p(x_0)$, $\lambda(x)\equiv\lambda(x_0)$ and $f\equiv 0$.

Thus, using Proposition \ref{identific-alpha}, we get that
$\alpha=\lambda^*(x_0)$.
\end{proof}

Our next result is

\begin{theo}\label{limsup-bola} Let $p, f, \lambda$ and $u$ be as in Theorem \ref{limsup-grad}.
 Let $x_0\in  \Omega\cap\partial\{u>0\}$. Assume there is  a ball $B$ contained in $ \{u=0\} $ touching  $ x_0$, then
\begin{equation}\label{limsup224}
\limsup_{\stackrel{x\to x_0}{u(x)>0}} \frac{u(x)}{\mbox{dist}(x,B)}=  \lambda^*(x_0),\end{equation} where
$\lambda^*(x)=\Big(\frac{p(x)}{p(x)-1}\,\lambda(x)\Big)^{1/p(x)}$.
\end{theo}
\begin{proof}
Let $\ell$ be the finite limit on the left hand side of \eqref{limsup224} and let $y_k \to x_0$
with $u(y_k)>0$ be such that
$$\frac{u(y_k)}{d_k}\to \ell, \quad d_k=\mbox{dist}(y_k,B).$$
Consider the blow up sequence $u_k$ with respect to
$B_{d_k}(x_k)$, where $x_k\in\partial B$ are points with
$|x_k-y_k|=d_k$, that is, $u_k(x)=\frac{u(x_k+d_k x)}{d_k}$.
Choose a subsequence with blow up limit
$u_0$, such that there exists
$$e:=\lim_{k\to\infty} \frac{y_k-x_k}{d_k}.$$

Using Lemma \ref{lemm-min-ig} and Theorem \ref{nodegfinal} and proceeding as in the proof of Theorem 5.2 in \cite{LW4} we have that $u_0(x)=\ell\langle x, e\rangle^+$.
Thus, applying Propositions \ref{blow-up-mild-minim} and \ref{identific-alpha}, we get that
$\ell=\lambda^*(x_0)$.
\end{proof}

The last result in this section is

\begin{theo}
\label{asymptotic-devel} Let $p, f, \lambda$ and $u$ be as in Theorem \ref{limsup-grad}.
 Let $x_0\in \Omega\cap\partial\{u>0\}$ be such that $\fb$ has at
$x_0$ an  inward unit normal $\nu$ in the measure theoretic sense. Then,
$$
u(x)=\lambda^*(x_0)\langle x-x_0, \nu\rangle^+  + o(|x-x_0|),
$$
where $\lambda^*(x)=\Big(\frac{p(x)}{p(x)-1}\,\lambda(x)\Big)^{1/p(x)}$.
\end{theo}
\begin{proof}
Take
$u_{\lambda}(x)=\frac{1}{\lambda} u(x_0+\lambda x).$ Let $\rho>0$ such
that $B_{\rho}(x_0)\subset\subset\Omega$. Since $u_{\lambda}\in
Lip(B_{\rho/\lambda})$ uniformly in $\lambda$, $u_{\lambda}(0)=0$,
there exist $\lambda_j\to 0$ and $U$ such that
$u_{\lambda_j}\to U$ uniformly on compact sets of $\R^N$.
Since $|\nabla u(x)|\le L_0$ in $B_{r_0}(x_0)$ for some positive $L_0$ and $r_0$ then,
for any $M>0$, $|\nabla u_{\lambda_j}(x)|\le L_0$ in $B_M(0)$ for $j$ large. Therefore, $|\nabla U(x)|\le L_0$ in $\R^N$ and
$U\in Lip(\R^ N)$.

Without loss of generality we assume that $x_0=0$, and $\nu=e_1$.  {}From
Lemma \ref{lemm-min-ig}, $\Delta_{p(\lambda x)} u_{\lambda}=\lambda f(\lambda x)$
in $\{u_{\lambda}>0\}$. Using the fact that $e_1$ is the
 inward normal in the measure theoretic sense, we have, for
fixed $k$,
$$|\{u_{\lambda}>0\}\cap\{x_1<0\}\cap B_k|\to 0 \quad \mbox{ as } \lambda \to 0.$$
Hence, $U=0$ in $\{x_1<0\}$. Moreover, $U$ is nonnegative in
$\{x_1>0\}$, $\Delta_{p_0} U=0$ in $\{U>0\}$ with $p_0=p(x_0)$ and $U$ vanishes in $\{x_1\leq
0\}$. Then, by Lemma \ref{development11} we have that there
exists $\alpha\geq 0$ such that
$$U(x)=\alpha x_1^++o(|x|).$$

Define
$U_{\lambda}(x)=\frac{1}{\lambda} U(\lambda x)$, then
$U_{\lambda}\to \alpha x_1^+$ uniformly on compact sets of
$\R^N$.

Now, by Theorem \ref{nodegfinal} and Remark \ref{equiv-nondeg}, we have, for some $c>0$ and $0<r<r_0$,
$$\frac{1}{r^N} \int_{B_r} u_{\lambda_j} \, dx \geq cr$$ and then
$$\frac{1}{r^N} \int_{B_r} U_{\lambda_j} \, dx \geq cr.$$

Therefore $\alpha>0$. Now, since $u$ is a nonnegative local  minimizer of functional $J$ in $\Omega$, then $u$ is locally Lipschitz and  it is a nonnegative  mild minimizer of $J$ in $\Omega$. Thus, by Proposition \ref{blow-up-mild-minim}, $U$ is a nonnegative Lipschitz mild minimizer of $J$ in $\R^ N$ with $p(x)\equiv p(x_0)$, $\lambda(x)\equiv\lambda(x_0)$ and $f\equiv 0$. Then, applying Proposition
\ref{blow-up-mild-minim} to $U$ we get that $U_0=\alpha x_1^+$ is also a nonnegative Lipschitz mild minimizer of $J$ in $\R^ N$ with $p(x)\equiv p(x_0)$, $\lambda(x)\equiv\lambda(x_0)$ and $f\equiv 0$.

Now, by Proposition \ref{identific-alpha}, $\alpha=\lambda^*(x_0).$

We have shown that
$$U(x)=\begin{cases} \lambda^*(x_0) x_1+o(|x|) &\quad x_1>0\\
0 &\quad x_1\leq 0.
       \end{cases}
$$

Then, using that $\Delta_{p(\lambda x)} u_{\lambda}=\lambda
f(\lambda x)$ in $\{u_{\lambda}>0\}$, by interior H\"older
gradient estimates we have $\nabla u_{\lambda_j}\rightarrow \nabla
U$ uniformly on compact subsets of $\{U>0\}$. Then, by Theorem
\ref{limsup-grad}, $|\nabla U|\leq \lambda^*(x_0)$ in $\R^N$.
As $U=0$ on $\{x_1=0\}$ we have, $U\leq \lambda^*(x_0)x_1$ in
$\{x_1>0\}$.

Now, proceeding as in the proof of Theorem 5.3 in  \cite{LW4}, we conclude that  $U\equiv \lambda^*(x_0)x_1^+$ and the result follows.
\end{proof}

\end{section}
%%%%%%%%%%%end section energy minimizers AC%%%%%%%%%%%%%%%%%%%%%%%

%%%%%%%%%% section energy minimizers persing %%%%%%%%%%%%%%%%%%%%%%%%%%%%%%%%%%%%%%%%
\begin{section}{Energy minimizers of energy functional \eqref{funct-Jep}}
\label{sect-energy-minim-persing}

\smallskip
In this section  we prove existence of minimizers of the energy functional \eqref{funct-Jep} and, in the spirit of the previous section, we  develop an exhaustive analysis of the essential  properties  of functions $\uep$
which are  nonnegative local  minimizers of that energy. As a consequence we obtain results for
solutions $\uep$ to the singular perturbation problem $\pep(\fep,
p_{\ep})$ which  are nonnegative local energy minimizers and moreover, we  get
results for their limit functions $u$.

We start by pointing out that the same considerations in Definition \ref{def-loc-min} and Remarks \ref{rem-local} and \ref{rem-global} for functional \eqref{funct-J} apply to functional \eqref{funct-Jep} in the present section.

We first obtain

\begin{theo}\label{minimizers} Let $\Omega\subset\R^N$ be a bounded domain and let $\phi_\ep\in W^{1,\psubep(\cdot)}(\Omega)$
be such that $\|\phi_\ep\|_{1,\psubep(\cdot)}\le {\mathcal A}_1$,
with $1<p_{\min}\le p_{\ep}(x)\le p_{\max}<\infty$ and  $\|\nabla p_{\ep}\|_{L^{\infty}}\leq
L$.  Let $f^\ep\in L^\infty(\Omega)$ such that
$\|f^\ep\|_{L^\infty(\Omega)}\le {\mathcal A}_2$. There exists
$\uep\in W^{1,\psubep(\cdot)}(\Omega)$ that minimizes the energy
\begin{equation}\label{energy}
J_\ep(v)=\di \int_\Omega \Big(\frac{|\nabla
v|^{\psubep(x)}}{\psubep(x)}+B_\ep(v)+f^\ep v\Big)\, dx
\end{equation}
among functions $v\in W^{1,\psubep(\cdot)}(\Omega)$ such that $v-\phi_\ep \in W_0^{1,\psubep(\cdot)}(\Omega)$.
Here $B_\ep(s)=\int_0^s\beta_\ep(\tau)\,d\tau$.

Then, the function $\uep$ satisfies
\begin{equation}\label{ecuac-sin-signo}
 \Delta_{\psubep(x)}\uep=\beta_\ep(\uep)+f^\ep \quad\mbox{in}\quad\Omega
\end{equation}
and for every $\Omega'\subset\subset\Omega$ there exists
$C=C(\Omega', {\mathcal A}_1, {\mathcal A}_2, p_{\min}, p_{\max}, L)$ such that
\begin{equation}\label{cotsup}
\sup_{\Omega'} \uep\le C.
\end{equation}
\end{theo}
\begin{proof} Let us prove first that a minimizer exists. In fact, let
$$
\mathcal{K}^{\ep}=\Big\{v\in W^{1,\psubep(\cdot)}(\Omega)\colon v - \phi_\ep\in W_0^{1,\psubep(\cdot)}(\Omega)\Big\}.
$$
In order to prove that $J_\ep$ is bounded from below in $\mathcal{K}^{\ep}$,  we observe that if $v\in\mathcal{K}^{\ep}$, then
$$
J_\ep(v)\ge\frac{1}{p_{\max}}\int_\Omega |\nabla v|^{\psubep(x)}\,+ \int_\Omega f^\ep v\, dx,
$$
and we have, by Theorem \ref{holder} and Theorem \ref{poinc},
\begin{align*}
\int_\Omega |f^\ep v|\, dx &\le 2\|\fep\|_{{\psubep}'(\cdot)}\|v\|_{\psubep(\cdot)}\le 2\|\fep\|_{{\psubep}'(\cdot)}(\|v-\phi_\ep\|_{\psubep(\cdot)}
+\|\phi_\ep\|_{\psubep(\cdot)})\\
&\le C_0\|\nabla v-\nabla\phi_\ep\|_{\psubep(\cdot)}+C_1\le  C_0\|\nabla
v\|_{\psubep(\cdot)}+C_2.
\end{align*}
If  $\Big(\int_{\Omega} |\nabla v|^{\psubep(x)}\, dx\Big)^{1/{p_{\min}}}\ge \Big(\int_{\Omega} |\nabla v|^{\psubep(x)}\, dx\Big)^{1/{p_{\max}}}$ we get,
by Proposition \ref{equi},
$$
\int_\Omega |f^\ep v|\, dx \le  C_0\Big(\int_{\Omega} |\nabla
v|^{\psubep(x)}\, dx\Big)^{1/{p_{\min}}}+C_2\le C_3 + \frac{1}{2\,
p_{\max}}\int_\Omega |\nabla v|^{\psubep(x)}\, dx.
$$
If, on the other hand, $\Big(\int_{\Omega} |\nabla v|^{\psubep(x)}\,
dx\Big)^{1/{p_{\min}}}< \Big(\int_{\Omega} |\nabla v|^{\psubep(x)}\,
dx\Big)^{1/{p_{\max}}}$, we get in an analogous way
$$
\int_\Omega |f^\ep v|\, dx \le  C_0\Big(\int_{\Omega} |\nabla
v|^{\psubep(x)}\, dx\Big)^{1/{p_{\max}}}+C_2\le  C_4 + \frac{1}{2\,
p_{\max}}\int_\Omega |\nabla v|^{\psubep(x)}\, dx.
$$
Taking $C_5=\max\{C_3, C_4\}$, we get
\begin{equation}\label{cotaJ}
J_\ep(v)\ge -C_5 + \frac{1}{2\, p_{\max}}\int_\Omega |\nabla
v|^{\psubep(x)}\, dx\ge -C_5,
\end{equation}
which shows that $J_\ep$ is bounded from below in $\mathcal{K}^{\ep}$.

At this point we want to remark that the constants $C_0,...,C_5$ above can be taken depending only on
${\mathcal A}_1, {\mathcal A}_2, p_{\min}, p_{\max}$ and $L$.

We now take  a minimizing sequence $\{u_n\}\subset\K^{\ep}$. Without loss of generality we can assume that $J_\ep(u_n)\le J_\ep(\phi_\ep)$, so
by  \eqref{cotaJ},$\int_{\Omega}|\nabla u_n|^{\psubep(x)}\le C_6$. By Proposition \ref{equi}, $\|\nabla
u_n-\nabla \phi_\ep\|_{\psubep(\cdot)}\leq C_7$ and, as  $u_n - \phi_\ep\in W_0^{1,\psubep(\cdot)}(\Omega)$, by Theorem \ref{poinc} we
 have $\|u_n-\phi_\ep\|_{\psubep(\cdot)}\leq C_8$. Therefore, by Theorem
\ref{ref}
 there exist a
subsequence (that we still call $u_n$) and a function $\uep\in
W^{1,\psubep(\cdot)}(\Omega)$ such that
\begin{equation}
\label{cotanorma} ||\uep||_{W^{1,\psubep(\cdot)}(\Omega)}\le {\bar C},\quad \mbox{ with } {\bar C}={\bar C}({\mathcal A}_1, {\mathcal A}_2, p_{\min}, p_{\max},L),
\end{equation}
$$
u_n \rightharpoonup  \uep \quad \mbox{weakly in } W^{1,\psubep(\cdot)}(\Omega),
$$
and, by Theorem \ref{imb},
\begin{align*}
&  u_n \rightharpoonup  \uep \quad \mbox{weakly in }
W^{1,p_{\min}}(\Omega).
\end{align*}
Now, by the compactness of the immersion
$W^{1,{p_{\min}}}(\Omega)\hookrightarrow L^{p_{\min}}(\Omega)$ we
have that, for a subsequence that we still denote by $u_n$,
\begin{align*}
u_n &\to \uep \quad \mbox{in }L^{p_{\min}}(\Omega),\\
u_n &\to \uep \quad \mbox{a.e. } \Omega.
\end{align*}

As $\K^{\ep}$ is convex and closed, it is weakly closed, so $\uep\in \K^{\ep}$.

It follows that
\begin{align*}
\lim_{n\to\infty} \int_{\Omega}B_{\ep}(u_n)\, dx &=  \int_{\Omega}B_{\ep}(\uep)\, dx,\\
\lim_{n\to\infty} \int_{\Omega}\fep u_n\, dx &=  \int_{\Omega} \fep\uep\, dx,\\
 \int_{\Omega} \frac{|\nabla \uep|^{\psubep(x)}}{\psubep(x)}\, dx &\le \liminf_{n\to\infty}
\int_{\Omega}\frac{|\nabla u_n|^{\psubep(x)}}{\psubep(x)}\, dx.
\end{align*}
In order to prove the last inequality we proceed as in \eqref{G-AC} in Theorem \ref{existence-minimizers-AC}.

Hence
$$
J_\ep(\uep)\le \liminf_{n\to\infty}J_\ep(u_n) =
\inf_{v\in\K^{\ep}} J_\ep(v).
$$
Therefore, $\uep$ is a minimizer of $J_\ep$ in $\K^{\ep}$.

Let us now prove that there holds \eqref{ecuac-sin-signo}.
Let $t>0$ and $\xi\in C^{\infty}_0(\Omega)$. Using the minimality
of $\uep$  we have
\begin{align*}0 &\leq \frac{1}{t} (J_{\ep}(\uep-t
\xi)-J_{\ep}(\uep))= \frac{1}{t} \int_{\Omega}\Big(\frac{|\nabla
\uep-t \nabla \xi|^{\psubep(x)}}{\psubep(x)}-\frac{|\nabla \uep|^{\psubep(x)}}{\psubep(x)}\Big)\, dx \, +\\
&\frac{1}{t}\int_{\Omega}\Big(B_{\ep}(\uep-t\xi)-B_{\ep}(\uep)\Big)\, dx
+\frac{1}{t}\int_{\Omega}\Big(\fep(\uep-t\xi)-\fep\uep\Big)\, dx\\
& \leq -\int_{\Omega}
 |\nabla \uep-t \nabla \xi|^{\psubep(x)-2}(\nabla \uep-t \nabla
\xi)\cdot \nabla \xi\, dx
+\frac{1}{t}\int_{\Omega}\Big(B_{\ep}(\uep-t\xi)-B_{\ep}(\uep)\Big)\, dx
-\int_{\Omega}\fep\xi\, dx
\end{align*}
and if
we take $t\rightarrow 0$, we obtain
\begin{equation}\label{unsigno}
0\leq -\int_{\Omega}
|\nabla \uep|^{\psubep(x)-2}\nabla \uep\cdot\nabla \xi\, dx-\int_{\Omega}\beta_{\ep}(\uep)\xi\, dx-\int_{\Omega}\fep\xi\, dx.
\end{equation}
If we now take $t<0$, and proceed in a similar way, we obtain the opposite sign in \eqref{unsigno} and \eqref{ecuac-sin-signo} follows.

Finally, in order to prove \eqref{cotsup}, we observe that, from Proposition \ref{equi} and estimate \eqref{cotanorma},
we have that $\int_{\Omega}|\uep|^{\psubep(x)}\, dx \le {\bar C_1}({\mathcal A}_1, {\mathcal A}_2, p_{\min}, p_{\max},L)$. Thus, the desired estimate
follows from the application of Proposition 2.1 in \cite{Wo}, since $\Delta_{\psubep(x)}\uep\ge f^\ep \ge -{\mathcal A}_2$ in $\Omega$.
\end{proof}

\begin{rema}\label{rema-signo}
We are interested in studying the behavior of a family $\uep$ of nonnegative local minimizers of the energy $J_{\ep}$ defined in \eqref{energy}.

If $u_{\ep}$ are as in Theorem \ref{minimizers} then $u_{\ep}$ satisfy \eqref{ecuac-sin-signo} and it follows from Proposition 2.1 in \cite{Wo}
that $u_{\ep}\in L^\infty_{\rm loc}(\Omega).$ Moreover, by Theorem 1.1 in \cite{Fan} $u_{\ep}\in C^1(\Omega)$ and $\nabla u_{\ep}$ are locally
H\"older continuous in $\Omega$.

If we have, for instance, that $\phi_\ep\ge 0$ in $\Omega$ and $\fep\le 0$ in $\Omega$,
then we have $\uep\ge 0$ in $\Omega$. In fact, the result follows by observing that, for every $\ep>0$,
${\xi}^{\ep}=\min (u_{\ep},0)\in W_0^{1,\psubep(\cdot)}(\Omega)$. Then,  we get \eqref{unsigno} for the test function ${\xi}^{\ep}$
and, using that $\beta_{\ep}(\uep){\xi}^{\ep}=0$ and $\fep\le 0$, we obtain $\int_{\Omega} |\nabla {\xi}^{\ep}|^{\psubep(x)}\,dx=0$, which implies
$\uep\ge 0$ in $\Omega$.
\end{rema}

\begin{rema}\label{rema-apply}
Let $\uep$ be a family of nonnegative local minimizers of the energy
 functional $J_{\ep}$  defined in \eqref{energy} which are uniformly bounded, with $\fep$ and $p_\ep$ uniformly bounded
 (like for instance the one constructed in Theorem \ref{minimizers}
and Remark \ref{rema-signo}). Then, as in Theorem \ref{minimizers} we deduce  that   $\uep$ are solutions to
$P_\ep(\fep,p_\ep)$ and thus,  all the results in our work \cite{LW4} apply
to this family. In particular, there hold the local uniform gradient
estimates of Theorem 2.1 in \cite{LW4} and  the results on passage to
the limit in Lemma 3.1 in \cite{LW4}.
\end{rema}

We also have

\begin{theo}\label{limit-minim}
Assume that $1<p_{\min}\le \psubepj(x)\le p_{\max}<\infty$ and that
$\|\nabla \psubepj\|_{L^{\infty}}\leq L$.
 Let $\uepj\in W^{1,\psubepj(\cdot)}(\Omega)$ be nonnegative local minimizers of
\begin{equation}\label{ener-epj}
J_{\ep_j}(v)=\di \int_\Omega \Big(\frac{|\nabla
v|^{\psubepj(x)}}{\psubepj(x)}+B_{\ep_j}(v)+\fepj v\Big)\, dx,
\end{equation}
with $\|u^{\ep_j}\|_{L^{\infty}(\Omega)}\leq L_1$
 and $\|\fepj\|_{L^\infty(\Omega)}\le L_2$, such that $u^{\ep_j}\rightarrow u$ uniformly on compact subsets of
$\Omega$, $\fepj\rightharpoonup f$ $*-$weakly in
$L^\infty(\Omega)$, $p_{\epj}\to p$ uniformly on compact subsets
of $\Omega$  and $\ep_j\to 0$. Then, $u$ is locally Lipschitz. Let
$B_r=B_r(x_0)\subset\subset\Omega$ and denote
\begin{align}
 J(v)=\int_\Omega\Big(\frac{|\nabla
 v|^{p(x)}}{p(x)}+M\chi_{\{v>0\}}+fv\Big)\,dx,\label{energy-funct}\\
 J_{r,0}(v)=\int_{B_r}\Big(\frac{|\nabla  v|^{p(x)}}{p(x)}+M\chi_{\{v>0\}}+fv\Big)\,dx,\end{align}
where $M=\int \beta(s)\, ds$.
\item{i)} If $v\in W^{1,p(\cdot)+\delta}(B_r)$ for some $\delta>0$ and $v-u\in W_0^{1,p(\cdot)}(B_r)$, then $J_{r,0}(u)\le J_{r,0}(v)$.
\item{ii)} If there holds that $p_{\epj}\le p$ in $\Omega$ and $u\in W^{1,p(\cdot)}(\Omega)$, then $u$ is a nonnegative local minimizer of functional \eqref{energy-funct}.
\end{theo}
\begin{proof} We first observe that the estimates of Theorem 2.1 in \cite{LW4} apply, as well as the results in Lemma 3.1 in \cite{LW4}. In particular,
 $\uepj$ are locally uniformly Lipschitz and therefore $u$ is locally Lipschitz in
$\Omega$.

We will follow the ideas in Theorem 1.16 in \cite{CS}. In fact, let $B_r=B_r(x_0)\subset\subset\Omega$, for simplicity assume $x_0=0$, and denote
\begin{align*}
J_{r,j}(v)= \int_{B_r} \Big(\frac{|\nabla v|^{\psubepj(x)}}{\psubepj(x)}+B_{\ep_j}(v)+\fepj v\Big)\, dx,\\
J_{r,0}(v)=\int_{B_r}\Big(\frac{|\nabla  v|^{p(x)}}{p(x)}+M\chi_{\{v>0\}}+fv\Big)\,dx.
\end{align*}

Let us first assume that ii) holds.

Given $v\in W^{1,p(\cdot)}(B_r)$ such that $v-u\in W_0^{1,p(\cdot)}(B_r)$, we want to show that
\begin{equation}\label{minimiza}
J_{r,0}(u)\le J_{r,0}(v).
\end{equation}
For $h>0$ small, we define
$$
v_{h,j}=
\begin{cases}
v & \mbox{in }B_r,\\
u+\frac{|x|-r}{h}(\uepj-u) & \mbox{in }B_{r+h} \setminus B_r.
\end{cases}
$$
Then, since $p_{\epj}\le p$, it follows that $v_{h,j}\in W^{1,p_{\epj}(\cdot)}(B_{r+h})$, $v_{h,j}-\uepj\in W_0^{1,p_{\epj}(\cdot)}(B_{r+h})$ and there
holds
\begin{align*}
&J_{r+h,j}(v_{h,j})= \int_{B_{r+h}} \Big(\frac{|\nabla v_{h,j}|^{\psubepj(x)}}{\psubepj(x)}+B_{\ep_j}(v_{h,j})+\fepj v_{h,j}\Big)\le J_{r,0}(v)\\
&\qquad+\int_{B_{r+h}\setminus B_r} \Big(\frac{|\nabla v_{h,j}|^{\psubepj(x)}}{\psubepj(x)}+B_{\ep_j}(v_{h,j})+\fepj v_{h,j}\Big)\
+\int_{B_r}\Big( \frac{|\nabla v|^{\psubepj(x)}}{\psubepj(x)}-\frac{|\nabla  v|^{p(x)}}{p(x)}+(\fepj-f)v\Big) \\
&\quad\le J_{r,0}(v)+ C_0hr^{N-1} + C_1\int_{B_{r+h}\setminus B_r} \frac{|\uepj-u|^{\psubepj(x)}}{h^{\psubepj(x)}}+
\int_{B_r}\Big( \frac{|\nabla v|^{\psubepj(x)}}{\psubepj(x)}-\frac{|\nabla  v|^{p(x)}}{p(x)}+(\fepj-f)v\Big).
\end{align*}
Therefore,
\begin{equation}\label{limsupvhj}
\limsup_{j\to\infty}J_{r+h,j}(v_{h,j})\le J_{r,0}(v)+ C_0hr^{N-1}.
\end{equation}
On the other hand,
$$
M\chi_{\{u>0\}}\le\liminf_{j\to\infty} B_{\epj}(\uepj),
$$
which implies
\begin{equation}\label{primterm}
\int_{B_r} M\chi_{\{u>0\}}\,dx \le
\liminf_{j\to\infty} \int_{B_r}B_{\epj}(\uepj)\, dx.
\end{equation}
In addition, since
$\nabla \uepj \rightharpoonup \nabla u$ weakly  in
$L^{p(\cdot)}({B_r})$,
arguing in a similar way as in Theorem \ref{minimizers}, we get
\begin{equation}\label{segterm}
 \int_{B_r} \frac{|\nabla u|^{p(x)}}{p(x)}\, dx \le \liminf_{j\to\infty}
\int_{B_r}\frac{|\nabla \uepj|^{p(x)}}{p(x)}\, dx=\liminf_{j\to\infty}
\int_{B_r}\frac{|\nabla \uepj|^{\psubepj(x)}}{\psubepj(x)}\, dx.
\end{equation}
Now, using \eqref{primterm} and \eqref{segterm}, and the fact that
$\uepj$ are nonnegative local minimizers of $J_{\epj}$,  we obtain
$$
J_{r,0}(u)\le \liminf_{j\to\infty} J_{r,j}(\uepj)\le \liminf_{j\to\infty} J_{r+h,j}(\uepj)+ C_2hr^{N-1}
\le\liminf_{j\to\infty} J_{r+h,j}(v_{h,j})+ C_2hr^{N-1},
$$
which in combination with \eqref{limsupvhj} gives
$$
J_{r,0}(u)\le J_{r,0}(v)+ C_3hr^{N-1}.
$$
Therefore, letting $h\to 0$ we obtain \eqref{minimiza}.

Finally, if there holds i) we can proceed exactly as above to
prove that \eqref{minimiza} holds, using that in this case we also
have $v_{h,j}\in W^{1,p_{\epj}(\cdot)}(B_{r+h})$,
$v_{h,j}-\uepj\in W_0^{1,p_{\epj}(\cdot)}(B_{r+h})$  for large
$j$.
\end{proof}

\begin{rema}\label{nodeg-dist-min}
 Let $\uep$ be a family of nonnegative local minimizers of
$J_\ep(v)=\int_\Omega \big( \frac{|\nabla v|^{\psubep(x)}}{\psubep(x)}+B_\ep(v)+f^\ep v\big)\, dx$,
with $1<p_{\min}\le p_{\ep}(x)\le p_{\max}<\infty$, $\|\nabla p_{\ep}\|_{L^{\infty}}\leq L$, $\|\uep\|_{L^{\infty}(\Omega)}\leq L_1$ and
$\|f^\ep\|_{L^\infty(\Omega)}\le L_2$.
Then, with a minor modification of the proof of Theorem \ref{nodegdist}, we can prove that, given $\Omega'\subset\subset\Omega$,
there  exist positive constants $c_0$ and $\rho$  such that, for every $x_0\in\Omega'$,
\begin{equation*}
 \uep>\ep \quad \mbox{ in } B_{d_0}(x_0) \mbox{ with } \ 0<d_0\le \rho, \quad \mbox{ implies } \uep(x_0)\ge c_0 d_0,
  \end{equation*}
and, in particular,
\begin{equation*}
 \uep(x_0)\ge c_0 {\rm dist}(x_0,\{\uep\le \ep\}), \quad \mbox{ if } \ {\rm dist}(x_0,\{\uep\le \ep\})\le \rho,
  \end{equation*}
with $c_0$ and $\rho$ depending only on $p_{\min}, p_{\max}, L, L_1, L_2, M=\int\beta(s)ds$ and ${\rm dist}(\Omega',\partial\Omega)$.

As a consequence it follows that, if $u=\lim\uepj$ as $\epj\to 0$ then, for every $x_0\in\Omega'$,
\begin{equation*}
 u(x_0)\ge c_0 {\rm dist}(x_0,\{u\equiv 0\}), \quad \mbox{ if } \ {\rm dist}(x_0,\{u\equiv 0\})\le \rho.
  \end{equation*}
\end{rema}

\medskip

As in the case of minimizers of the energy \eqref{funct-J}, for
minimizers of the singular perturbation problem we have

\begin{theo}\label{nodegfinal-sing} Let $p_{\epj}$, $\fepj$, $\uepj$,  $\ep_j$, $p$,  $f$ and $u$   be as
in Theorem \ref{limit-minim}.
Let $\Omega'\subset\subset\Omega$. There exist  constants  $c>0$, $r_0>0$ such that if  $x_0\in \Omega'\cap\partial\{u>0\}$ and $r\le r_0$ then
$$\sup_{B_{r}(x_0)} u\geq c r.$$
The constants depend only  on N, $p_{\min}, p_{\max}, L, L_1, L_2, M, ||\beta||_{L^{\infty}}$ and ${\rm dist}(\Omega',\partial\Omega)$.
\end{theo}
\begin{proof}
The proof follows as that of Theorem \ref{nodegfinal}, replacing Theorem \ref{nodegdist} by Remark \ref{nodeg-dist-min}.
\end{proof}

In an analogous way as we obtained for minimizers of functional
\eqref{funct-J}, for minimizers of the singular perturbation
problem we have

\begin{theo}\label{densprop-sing} Let $p_{\epj}$, $\fepj$, $\uepj$,  $\ep_j$, $p$,  $f$ and $u$   be as
in Theorem \ref{limit-minim}.
 Let $\Omega'\subset\subset\Omega$. There exist constants $\tilde{c}\in (0,1)$ and ${\tilde r}_0>0$  such that, if $x_0\in\Omega'\cap\partial\{u>0\}$ with
$B_r(x_0)\subset \Omega'$ and $r\le {\tilde r}_0$, there holds
\begin{equation*}
 \frac{|B_r(x_0)\cap\{u>0\}|}{|B_r(x_0)|} \leq 1-\tilde{c}.
  \end{equation*}
The constants depend only on N, $p_{\min}, p_{\max}, L, L_1, L_2, M, ||\beta||_{L^{\infty}}$ and ${\rm dist}(\Omega',\partial\Omega)$.
 \end{theo}
\begin{proof}
The proof follows as that of Theorem \ref{densprop}. In this case we obtain estimate \eqref{baruk-min} by using part i) in Theorem \ref{limit-minim},
 since $v_k\in W^{1,{\bar p}_k(\cdot)+{\delta}_k}(B_{1/2})$, for some ${\delta}_k>0$ (see, for instance, Lemma 4.1 in \cite{Fan}).
\end{proof}

\end{section}
%%%%%%%%%%%end section energy minimizers persing%%%%%%%%%%%%%%%%%%%%%%%

%%%%%%%%%begin section regularity of the free boundary%%%%%%%%%%%%%%%%%%%%%%%%%%%%%%%%%
\begin{section}{Regularity of the Free Boundary}
\label{sect-regul-fb}

In this section, we first consider nonnegative local minimizers to
the energy functional \eqref{funct-J} and we obtain results on the
regularity of the free boundary  for these functions, which are a
consequence of the results in Section \ref{sect-energy-minim-ac}
and the results  in our work \cite{LW5}.

In addition,  we consider any family $\uep$ of nonnegative local
minimizers to the energy functional \eqref{funct-Jep} which are
uniformly bounded, with $\fep$ and $p_{\ep}$ uniformly bounded
(like, for instance, the one constructed in Theorem
\ref{minimizers} and Remark \ref{rema-signo}). Then (recall Remark
\ref{rema-apply}), all the results in our previous paper
\cite{LW4}  apply to such a family. Hence, as a consequence of the
results in Section \ref{sect-energy-minim-persing} and in our work
\cite{LW5}, we obtain results on the regularity of the free
boundary for limit functions of this family.

\medskip

First, for nonnegative local minimizers to the energy functional
\eqref{funct-J}, we get

\begin{theo}\label{AC=weak}  Assume that
 $1<p_{\min}\le p(x)\le p_{\max}<\infty$ with
$\|\nabla p\|_{L^{\infty}}\leq L$, $f\in L^{\infty}(\Omega)$ and $0<\lone\le\lambda(x)\le\ltwo<\infty$ with $\lambda\in C(\Omega)$.
Let $u\in W^{1,p(\cdot)}(\Omega)\cap L^{\infty}(\Omega)$   be a nonnegative local minimizer of
 \eqref{funct-J} in a domain $\Omega\subset \Bbb R^{N}$.

Then,  $u$ is a weak solution to the free boundary problem:
$u\ge0$ and
\begin{equation}
\tag{$P(f,p,{\lambda}^*)$}
\begin{cases}
\Delta_{p(x)}u = f & \mbox{in }\{u>0\}\\
u=0,\ |\nabla u| = \lambda^*(x) & \mbox{on }\partial\{u>0\}
\end{cases}
\end{equation}
with $\lambda^*(x)=\Big(\frac{p(x)}{p(x)-1}\,\lambda(x)\Big)^{1/p(x)}$.
\end{theo}

\begin{proof}
The result follows by applying Lemma  \ref{lemm-min-ig}, Corollary \ref{Lip} and
Theorems \ref{pre-lip}, \, \ref{nodegfinal}, \,  \ref{limsup-grad}, \, \ref{limsup-bola} and \ref{asymptotic-devel}.
\end{proof}

Now, we can apply the results in \cite{LW5} and deduce

\begin{theo}\label{regAC}
Let $p$, $f$, $\lambda$ and $u$ be as
in Theorem \ref{AC=weak}. Assume moreover that $f\in W^{1,q}(\Omega)$, $p\in W^{2,q}(\Omega)$
 with $q>\max\{1,N/2\}$ and $\lambda$ is H\"older continuous in $\Omega$.

Then, there is a subset $\mathcal{R}$  of the free boundary
$\Omega\cap\partial\{u>0\}$
$(\mathcal{R}=\partial_{\rm{red}}\{u>0\})$ which is locally a
$C^{1,\alpha}$ surface, for some $0<\alpha<1$, and  the free
boundary condition is satisfied in the classical sense in a
neighborhood of $\mathcal{R}$. Moreover, $\mathcal{R}$ is open and
dense in $\Omega\cap\partial\{u>0\}$ and the remainder of the free
boundary has $(N-1)-$dimensional Hausdorff measure zero.

If moreover $\nabla p$ and $f$ are  H\"older
continuous in $\Omega$, then the
equation is satisfied in the classical sense in a neighborhood of $\mathcal{R}$.
\end{theo}
\begin{proof}
We first observe that, by Theorem  \ref{AC=weak},  Theorem 4.4 in
\cite{LW5} applies at every
$x_0\in\Omega\cap\partial_{\rm{red}}\{u>0\}$.

Finally we observe that, since $u$ is a weak solution to
$P(f,p,{\lambda}^*)$, Theorem 2.1 in \cite{LW5} and Lemma 2.3 in
\cite{LW5} apply to $u$. Therefore, recalling Theorem
\ref{densprop} we deduce, from Theorem 4.5.6(3) in \cite{F}, that
${\mathcal H}^{N-1}(\fb\setminus\partial_{\rm{red}}\{u>0\})=0$.
\end{proof}

We also obtain higher regularity from the application of Corollary
4.1 in \cite{LW5}

\begin{coro}\label{high-AC} Let $p$, $f$, $\lambda$ and $u$ be as in Theorem \ref{regAC}.
Assume moreover that $p\in C^2(\Omega)$, $f\in C^1(\Omega)$ and
$\lambda\in C^2(\Omega)$ then $\partial_{\rm{red}}\{u>0\}\in
C^{2,\mu}$ for every  $0<\mu<1$.

If $p\in C^{m+1,\mu}(\Omega)$, $f\in C^{m,\mu}(\Omega)$ and
$\lambda\in C^{m+1,\mu}(\Omega)$ for some $0<\mu<1$ and $m\ge1$,
then $\partial_{\rm{red}}\{u>0\}\in C^{m+2,\mu}$.

Finally, if $p$, $f$ and $\lambda$ are analytic, then
$\partial_{\rm{red}}\{u>0\}$ is analytic.
\end{coro}

Next, for minimizers of the energy functional \eqref{funct-Jep} we
obtain, as a consequence of the results in Section
\ref{sect-energy-minim-persing} and the results in \cite{LW4}

\begin{theo}\label{lim=weak} Assume that $1<p_{\min}\le
p_{\epj}(x)\le p_{\max}<\infty$ and $\|\nabla
p_{\epj}\|_{L^{\infty}}\leq L$.
  Let $u^{\ep_j}\in W^{1,\psubepj(\cdot)}(\Omega)$ be a family of nonnegative local minimizers of \eqref{ener-epj}
in a domain $\Omega\subset \Bbb R^{N}$ such that
$u^{\ep_j}\rightarrow u$ uniformly on compact subsets of $\Omega$,
$\fepj\rightharpoonup f$ $*-$weakly in $L^\infty(\Omega)$,
$p_{\epj}\to p$ uniformly on compact subsets of $\Omega$ and
$\ep_j\to 0$.

Then,  $u$ is a weak solution to the free boundary problem:
$u\ge0$ and
\begin{equation}
\tag{$P(f,p,{\lambda}^*)$}
\begin{cases}
\Delta_{p(x)}u = f & \mbox{in }\{u>0\}\\
u=0,\ |\nabla u| = \lambda^*(x) & \mbox{on }\partial\{u>0\}
\end{cases}
\end{equation}
with $\lambda^*(x)=\Big(\frac{p(x)}{p(x)-1}\,M\Big)^{1/p(x)}$ and
$M=\int \beta(s)\, ds$.
\end{theo}

\begin{proof}
The result follows by applying first  Remark \ref{rema-apply} and
Theorems \ref{nodegfinal-sing} and \ref{densprop-sing} and then,
Theorem 6.1 in \cite{LW4}.
\end{proof}

We can now apply   the results in \cite{LW5} and deduce

\begin{theo}\label{regminim} Let $p_{\epj}$, $\fepj$, $\uepj$,  $\ep_j$, $p$,  $f$ and $u$   be as
in Theorem \ref{lim=weak}. Assume moreover that $f\in
W^{1,q}(\Omega)$ and $p\in W^{2,q}(\Omega)$
 with $q>\max\{1,N/2\}$.

Then, there is a subset $\mathcal{R}$  of the free boundary
$\Omega\cap\partial\{u>0\}$
$(\mathcal{R}=\partial_{\rm{red}}\{u>0\})$ which is locally a
$C^{1,\alpha}$ surface, for some $0<\alpha<1$, and  the free
boundary condition is satisfied in the classical sense in a
neighborhood of $\mathcal{R}$. Moreover, $\mathcal{R}$ is open and
dense in $\Omega\cap\partial\{u>0\}$ and the remainder of the free
boundary has $(N-1)-$dimensional Hausdorff measure zero.

If moreover $\nabla p$ and $f$ are  H\"older continuous in
$\Omega$, then the equation is satisfied in the classical sense in
a neighborhood of $\mathcal{R}$.
\end{theo}
\begin{proof}
We first observe that, by Theorem  \ref{lim=weak},  Theorem 4.4 in
\cite{LW5} applies at every
$x_0\in\Omega\cap\partial_{\rm{red}}\{u>0\}$.

Finally we observe that, since $u$ is a weak solution to
$P(f,p,{\lambda}^*)$, Theorem 2.1 in \cite{LW5} and Lemma 2.3 in
\cite{LW5} apply to $u$. Therefore, recalling Theorem
\ref{densprop-sing} we deduce, from Theorem 4.5.6(3) in \cite{F},
that ${\mathcal
H}^{N-1}(\fb\setminus\partial_{\rm{red}}\{u>0\})=0$.
\end{proof}

We also obtain higher regularity from the application of Corollary
4.1 in \cite{LW5}

\begin{coro}\label{high-persing} Let $p$, $f$ and $u$ be as in Theorem \ref{regminim}.
Assume moreover that $p\in C^2(\Omega)$ and $f\in C^1(\Omega)$,
then $\partial_{\rm{red}}\{u>0\}\in C^{2,\mu}$ for every
$0<\mu<1$.

If $p\in C^{m+1,\mu}(\Omega)$ and $f\in C^{m,\mu}(\Omega)$ for
some $0<\mu<1$ and $m\ge1$, then $\partial_{\rm{red}}\{u>0\}\in
C^{m+2,\mu}$.

Finally, if $p$ and $f$ are analytic, then
$\partial_{\rm{red}}\{u>0\}$ is analytic.
\end{coro}

\medskip

\end{section}
%%%%%%%% end section regularity of the free boundary %%%%%%%%%%%%%%%%%%%%
%%%%%%%%%%%%%%%%%%%appendix%%%%%%%%%%%%%%%%%%%%%%%%%%%%%%%%%%%%%%%%%%%%%%%%%%%%%%%%

\appendix

\section{} \label{appA1}

\setcounter{equation}{0}

In Section 1 we included some preliminaries on Lebesgue and Sobolev spaces with variable exponent. For the sake of completeness we collect here some
additional
results on these spaces as well as some other results that are used throughout the paper.

\bigskip
\begin{prop}\label{equi}
There holds
\begin{align*}
\min\Big\{\Big(\int_{\Omega} |u|^{p(x)}\, dx\Big)
^{1/{p_{\min}}},& \Big(\int_{\Omega} |u|^{p(x)}\, dx\Big)
^{1/{p_{\max}}}\Big\}\le\|u\|_{L^{p(\cdot)}(\Omega)}\\
 &\leq  \max\Big\{\Big(\int_{\Omega} |u|^{p(x)}\, dx\Big)
^{1/{p_{\min}}}, \Big(\int_{\Omega} |u|^{p(x)}\, dx\Big)
^{1/{p_{\max}}}\Big\}.
\end{align*}
\end{prop}

\bigskip

Some important results for these spaces are

\begin{theo}\label{ref}
Let $p'(x)$ such that $$\frac{1}{p(x)}+\frac{1}{p'(x)}=1.$$ Then
$L^{p'(\cdot)}(\Omega)$ is the dual of $L^{p(\cdot)}(\Omega)$.
Moreover, if $p_{\min}>1$, $L^{p(\cdot)}(\Omega)$ and
$W^{1,p(\cdot)}(\Omega)$ are reflexive.
\end{theo}

\begin{theo}\label{imb}
Let $q(x)\leq p(x)$. If $\Omega$ has finite measure, then
 $L^{p(\cdot)}(\Omega)\hookrightarrow L^{q(\cdot)}(\Omega)$
continuously.
\end{theo}

We also have the following H\"older's inequality

\begin{theo} \label{holder}
Let $p'(x)$ be as in Theorem \ref{ref}. Then there holds
$$
\int_{\Omega}|f||g|\,dx \le 2\|f\|_{p(\cdot)}\|g\|_{p'(\cdot)},
$$
for all $f\in L^{p(\cdot)}(\Omega)$ and $g\in L^{p'(\cdot)}(\Omega)$.
\end{theo}

The following version of Poincare's inequality holds

\begin{theo}\label{poinc} Let $\Omega$ be bounded. Assume that $p(x)$ is log-H\"older continuous  in $\Omega$ (that is, $p$ has a modulus of continuity $\omega(r)=C(\log \frac{1}{r})^{-1}$). For
every $u\in W_0^{1,p(\cdot)}(\Omega)$, the inequality
$$
\|u\|_{L^{p(\cdot)}(\Omega)}\leq C\|\nabla u\|_{L^{p(\cdot)}(\Omega)}
$$
holds with a constant $C$ depending only on N, $\rm{diam}(\Omega)$ and the log-H\"older modulus of continuity of $p(x)$.
\end{theo}

\smallskip

For the proof of these results and more about these spaces, see \cite{DHHR},
\cite{KR}, \cite{RaRe}, \cite{HH} and the references therein.

\bigskip

We will also need

\begin{lemm}\label{development11} Let $1<p_0<+\infty$.
Let $u$ be  Lipschitz continuous in $\overline{B_1^+}$, $u\geq 0$
in $B_1^+$, $\Delta_{p_0} u=0$ in $\{u>0\}$ and $u=0$ on $\{x_N=0\}$. Then,
in $B_1^+$ $u$ has the asymptotic development
$$
u(x)=\alpha x_N+ o(|x|),
$$
with $\alpha\ge 0$.
\end{lemm}
\begin{proof}
See \cite{CLW1} for $p_0=2$, \cite{DPS} for $1<p_0<+\infty$ and \cite{MW} for a more general operator.
\end{proof}

%%%%%%%%%%%%% end appendix %%%%%%%%%%%%%%%%%%%%%%%%%%%%%%%%

%%%%%%%%%%%% begin references %%%%%%%%%%%%%%%%%%%%%%%%%%

%%%%%%%%% end of references %%%%%%%%%%%%%%%%%%%%%%%%%%%%%%%

\end{document}